\documentclass[12pt]{amsart}
\usepackage{amsmath, amsthm, amssymb}
\usepackage{amsaddr}
\usepackage{graphicx}
\usepackage{eepic}
\usepackage{pstricks}

\newtheorem{thm}{Theorem}[section]

\newtheorem{prop}[thm]{Proposition}

\newtheorem{cor}[thm]{Corollary}

\newtheorem{lem}[thm]{Lemma}
\newtheorem{conj}[thm]{Conjecture}

\numberwithin{equation}{section}
\def\pf{\noindent {\it Proof.} }

\def\P{{\mathcal P}}
\def\I{{\mathcal I}}
\def\N{{\mathcal N}}
\def\M{{\mathcal M}}
\def\Mcnk{{\mathcal M}_{n,k}^{(c)}}

\def\st{\mathop{st}}
\def\bl{{\rm bl}}

\def\Var{{\rm Var}}
\def\E{{\rm E}}
\def\Cov{{\rm Cov}}

\def\crol{{\mathrm cr^{(\ell)}}}
\def\crou{{\mathrm cr^{(u)}}}
\def\croc{{\mathrm cr^{(c)}}}

\def\la{\lambda}

\def\1{$\bf{1}$}
\def\0{$\bf{0}$}

\def\fig2part{ \begin{figure}[h!]
 \hspace{0.8cm}
{\psset{unit=0.65cm}
\begin{pspicture}(-2.5,0)(2.5,2.3)% cas (d)2
\psdots(0,2.5) \uput[ur](0,2.5){\scriptsize 1} \SpecialCoor
\psarc(0,0){2.5}{7}{173}
\psarc[linestyle=dotted,dotsep=0.8pt](0,0){2.5}{-5}{7}
\psarc[linestyle=dotted,dotsep=1pt](0,0){2.5}{172}{183}
\def\zigzagCa{\psline(2.5;0)(2.25;2)(2.5;4)(2.25;6)(2.5;8)(2.25;10)(2.5;12)(2.25;14)(2.5;16)(2.25;18)(2.5;20)}
\def\zigzagCaa{\psline(2.5;0)(2.25;2)(2.5;4)(2.25;6)(2.5;8)(2.25;10)(2.5;12)}
\def\zigzagCaaa{\psline(2.5;0)(2.25;2)(2.5;4)(2.25;6)(2.5;8)}
\def\zigzagCb{\psline(2.5;0)(1.90;2)(2.5;4)(1.90;6)(2.5;8)(1.90;10)(2.5;12)(1.90;14)(2.5;16)(1.90;18)(2.5;20)}
\rput{3}(0,0){\zigzagCaaa}\rput{20}(0,0){\zigzagCb}\rput{49}(0,0){\zigzagCa}
\rput{78}(0,0){\zigzagCb} \rput{107}(0,0){\zigzagCa}
\rput{136}(0,0){\zigzagCb} \rput{165}(0,0){\zigzagCaaa}
\def\clocheCa{\pscurve(2.5;0)(1.2;15)(1.2;23)(2.5;38)}
\def\clocheCb{\pscurve(2.5;0)(1.7;19)(2.5;38)}
\rput{11}(0,0){\clocheCa}\rput{69}(0,0){\clocheCa}\rput{127}(0,0){\clocheCa}
\rput{40}(0,0){\clocheCb}\rput{98}(0,0){\clocheCb}
\pscurve[linestyle=dashed, dash=2pt 1pt](2.5;20)(2;15)(1.9;3)
\pscurve[linestyle=dashed, dash=2pt 1pt](2.5;156)(2;162)(1.9;176)
\end{pspicture}}
 \hspace{1.5cm}
 {\psset{unit=0.8cm}
\begin{pspicture}(0,0)(7,2) % cas lineaire
 \def\zigzagA{\multirput(0,0)(.1,0){5}{\psline(0,0)(0.05,.4)(.1,0)}}
 \def\zigzagB{\multirput(0,0)(.1,0){5}{\psline(0,0)(0.05,.2)(.1,0)}}
 \def\clocheA{\psbezier(0.5,0)(0.5,0)(1.25,2.5)(2,0)}
 \def\clocheB{\psbezier(0.5,0)(0.5,0)(1.25,2)(2,0)}
 \multirput(0,0)(2,0){2}{\clocheA}
 \multirput(0,0)(2,0){3}{\zigzagA}
 \multirput(1,0)(2,0){2}{\zigzagB}
 \multirput(1,0)(2,0){2}{\clocheB}
  \rput(4.5,0){\psline[linestyle=dotted,dotsep=1.3pt,linewidth=1.5pt](0,0)(0.5,1)}
 \rput(5,0){\psline[linestyle=dashed, dash=2pt 1.3pt,linewidth=1.5pt](0,0)(0.3,0.55)}
\psline(-0.3,0)(7,0)
\psline[linestyle=dotted,dotsep=1.3pt](5.5,0.5)(6.5,0.5)
 \uput[dl](0,0){\small $1$}
 \psdots(0,0)
\end{pspicture}}
\caption{Sketch of linear and circular representation of a set
partition into 2 blocks}\label{fig:2partition}
\end{figure}}

\def\figpiab{\begin{figure}[h!]
 \hspace{0.4cm}
{\psset{unit=0.65cm}
\begin{pspicture}(-2.5,0)(2.5,2.3)% cas (d)2
\SpecialCoor \psarc(0,0){2.5}{7}{173}
\psarc[linestyle=dotted,dotsep=0.8pt](0,0){2.5}{-5}{7}
\psarc[linestyle=dotted,dotsep=1pt](0,0){2.5}{172}{183}

\psdots(2.5;116) \uput[ul](2.5;116){\scriptsize $2b$}
\psdots(2.5;80) \uput[ur](2.5;80){\scriptsize 1} \psdots(2.5;60)
\uput[ur](2.5;60){\scriptsize 2} \psdots(2.5;40)
\uput[ur](2.5;40){\scriptsize 3} \psdots(2.5;20)
\uput[ur](2.5;20){\scriptsize 4}

\def\chorda{\pscurve(2.5;0)(2.05;20)(2.5;40)}
\def\chordb{\pscurve[linestyle=dashed](2.5;0)(1.6;20)(2.5;40)}
\def\zigzagD{\psline(2.5;0)(2.25;2)(2.5;4)(2.25;6)(2.5;8)(2.25;10)(2.5;12)(2.25;14)(2.5;16)}
\rput{0}(0,0){\chorda}\rput{40}(0,0){\chorda}\rput{98}(0,0){\chorda}\rput{138}(0,0){\chorda}%\rput{148}(0,0){\chorda}
\rput{20}(0,0){\chordb}\rput{116}(0,0){\chordb}
\rput{82}(0,0){\zigzagD}
\pscurve[linestyle=dashed](2.5;60)(1.6;88)(2.5;116)
\pscurve[linestyle=dashed, dash=2pt 1pt](2.5;20)(2;15)(1.9;3)
\pscurve[linestyle=dashed, dash=2pt 1pt](2.5;156)(2;162)(1.9;176)
\end{pspicture}}
\hspace{1cm}
 {\psset{unit=1cm}
\begin{pspicture}(0,0)(6.7,2) % cas lineaire
\psline(-0.3,0)(2.5,0)
\psline[linestyle=dotted,dotsep=0.15pt](2.5,0)(4,0)
\psline(4,0)(6.3,0)
\def\arca{\pscurve(0,0)(0.5,0.9)(1,0)}
 \def\arcb{\pscurve[linestyle=dashed](0,0)(0.5,0.6)(1,0)}
 \multirput(0,0)(1,0){2}{\arca}
 \multirput(0.5,0)(1,0){2}{\arcb}
 \rput(2,0){\psline[linestyle=dashed, dash=2pt 1.3pt,linewidth=1.5pt](0,0)(0.4,0.8)}
  \rput(2.5,0){\psline[linestyle=dashed, dash=2pt 1.3pt,linewidth=1.5pt](0,0)(0.3,0.55)}
 \multirput(4.5,0)(1,0){1}{\arca}
 \multirput(4,0)(1,0){2}{\arcb}
 \rput(4.5,0){\psline[linestyle=dashed, dash=2pt 1.3pt,linewidth=1.5pt](0,0)(-0.4,0.8)}
  \rput(4,0){\psline[linestyle=dashed, dash=2pt 1.3pt,linewidth=1.5pt](0,0)(-0.3,0.55)}
  \psdots(0,0)\uput[dl](0,0){\small $1$}
  \psdots(0.5,0)\uput[dl](.5,0){\small $2$}
  \psdots(1,0)\uput[dl](1.0,0){\small $3$}
  \psdots(1.5,0)\uput[dl](1.5,0){\small $4$}
  \psdots(6,0)\uput[dr](6,0){\small $2b$}
\end{pspicture}}
\caption{Sketch of circular and linear representation of
$\pi(a,b)$}\label{fig:pi(a,b)}
\end{figure}}

\begin{document}

\title[Distribution of crossings in set partitions]{On the limiting distribution
 of some numbers of crossings in
set partitions}

%    author information
\author[A. Kasraoui]{Anisse Kasraoui$^{*}$}
\address{Fakult\"at f\"ur Mathematik, Universit\"at Wien,\\
Nordbergstrasse 15,A-1090 Vienna,Austria}
\email{anisse.kasraoui@univie.ac.at}
\thanks{$^{*}$ Research supported by the grant S9607-N13 from Austrian Science Foundation FWF
 in the framework of the National Research Network  ``Analytic Combinatorics and Probabilistic Number theory".}

\maketitle

\begin{abstract}
 We study the asymptotic distribution of the two following combinatorial parameters:
the number of arc crossings in the linear representation, $\crol$,
and the number of chord crossings in the circular representation,
$\croc$, of a random set partition.  We prove that, for $k\leq
n/(2\,\log n)$ (resp., ${k=o(\sqrt{n})}$), the distribution of the parameter $\crol$ (resp., $\croc$) 
taken over partitions of $[n]:=\{1,2,\ldots,n\}$ into $k$
blocks is, after standardization, asymptotically Gaussian as $n$
tends to infinity. We give exact and asymptotic formulas for the
variance of the distribution of the parameter~$\crol$ from which we
deduce that the distribution of  $\crol$ and $\croc$ taken over all
partitions of $[n]$ is concentrated around its mean.  
The proof of these results relies on a standard analysis  of generating functions
associated with the parameter $\crol$ obtained in earlier work of
Stanton, Zeng and the author. We also
determine the maximum values of the parameters $\crol$ and $\croc$.
\end{abstract}

\section{Introduction and Main results}\label{sec:in}

\subsection{Introduction}

Take $n$ points on a circle labeled $1,2,\ldots,n$ and join them in
vertex disjoint simple polygons (the vertices of which are the
points $1,2,\ldots,n$). The resulting configuration is the
\emph{circular representation} of the set partition of
${[n]:=\{1,2,\ldots,n\}}$ the blocks of which consist of the
elements in the same polygons.  Alternatively, take $n$ points
labeled $1,2,\ldots,n$ on a line and join them in vertex disjoint
directed paths consisting of arcs oriented to the right drawn in the
upper half-plane. Arcs are always drawn in such a way  that any two arcs  cross
at most once. The resulting configuration is the \emph{linear
representation} of the set partition of $[n]$ the blocks of which
consist of the elements in the same paths.  An illustration is  given in
Figure~\ref{fig:representations}.
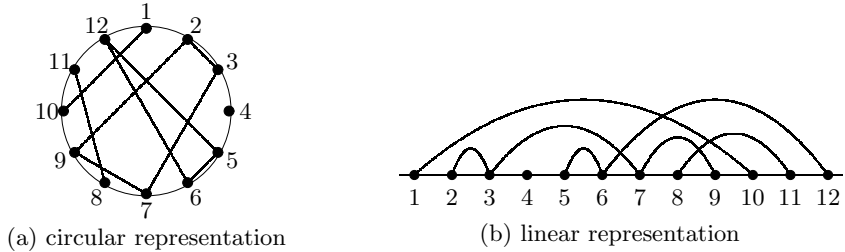
\begin{figure}[h]
\begin{center}
%%%%%%%%%%%%%%%%%%%%%%%% circular representation %%%%%%%%%%%%%%%%%%%%%%%%%%%%%%%%%%%%%%%%%%%%%%%%%%%%%%%%%%%%%%%%%%%%%%%%%%%%%%%%%%%%%%%%%%%%%%%%%%
%%%%%%%%%%%%%%%%%%%%%%%%%%%%%%%%%%%%%%%%%%%%%%%%%%%%%%%%%%%%% %%%%%%%%%%%%%%%%%%%%%%%% %%%%%%%%%%%%%%%%%%%%%%%%%%%%%%%%%%%%%%%%%%%%%%%
{ \setlength{\unitlength}{1.1mm}
\begin{picture}(25,20)(-10,-5)
 \put(0, 0){\circle{20,5}} \put(0, 10){\circle*{1,2}} \put(5, 8.66){\circle*{1,2}}
 \put(8.66, 5){\circle*{1,2}} \put(10, 0){\circle*{1,2}} \put(8.66, -5){\circle*{1,2}}
 \put(5, -8.66){\circle*{1,2}} \put(0,-10){\circle*{1,2}} \put(-5, -8.66){\circle*{1,2}}
 \put(-8.66, -5){\circle*{1,2}} \put(-10, 0){\circle*{1,2}} \put(-8.66, 5){\circle*{1,2}}
 \put(-5, 8.66){\circle*{1,2}}
 \put(0, 12){\makebox(0,0)[c]{\scriptsize 1}} \put(6, 10.39){\makebox(0,0)[c]{\scriptsize 2}}
 \put(10.39, 6){\makebox(0,0)[c]{\scriptsize 3}} \put(12,0){\makebox(0,0)[c]{\scriptsize 4}}
 \put(10.39,-6){\makebox(0,0)[c]{\scriptsize 5}} \put(6,-10.39){\makebox(0,0)[c]{\scriptsize 6}}
 \put(0,-12){\makebox(0,0)[c]{\scriptsize 7}} \put(-6,-10.39){\makebox(0,0)[c]{\scriptsize 8}}
 \put(-10.39,-6){\makebox(0,0)[c]{\scriptsize 9}} \put(-12,0){\makebox(0,0)[c]{\scriptsize 10}}
 \put(-10.39,6){\makebox(0,0)[c]{\scriptsize 11}} \put(-6,10.39){\makebox(0,0)[c]{\scriptsize 12}}
 %bloc1
 \qbezier(-10,0)(-5, 5)(0,10)
 %bloc2
 \qbezier(5,8.66)(6.83, 6.83)(8.66,5) \qbezier(8.66,5)(4.33, -2.5)(0,-10)
 \qbezier(0,-10)(-4.33,-7.5)(-8.66,-5) \qbezier(-8.66,-5)(-1.83,1.83)(5,8.66)
 %bloc3
 \qbezier(8.66,-5)(6.83,-6.83)(5,-8.66) \qbezier(5,-8.66)(0,0)(-5,8.66)\qbezier(-5,8.66)(1.83,1.83)(8.66,-5)
 %bloc4
 \qbezier(-5, -8.66)(-6.83, -1.83)(-8.66, 5)
\put(0,-15){\makebox(0,-1)[c]{\scriptsize (a) circular
representation}}
\end{picture}}
%%%%%%%%%%%%%%%%%%%%%%%% Linear representation %%%%%%%%%%%%%%%%%%%%%%%%%%%%%%%%%%%%%%%%%%%%%%%%%%%%%%%%%%%%%%%%%%%%%%%%%%%%%%%%%%%%%%%%%%%%%%%%%%
%%%%%%%%%%%%%%%%%%%%%%%%%%%%%%%%%%%%%%%%%%%%%%%%%%%%%%%%%%%%% %%%%%%%%%%%%%%%%%%%%%%%% %%%%%%%%%%%%%%%%%%%%%%%%%%%%%%%%%%%%%%%%%%%%%%%
\hspace{1.5cm} {\setlength{\unitlength}{1mm}
\begin{picture}(55,20)(0,3)
\put(-2,0){\line(1,0){59}}
\put(0,0){\circle*{1,2}}\put(0,0){\makebox(0,-6)[c]{\scriptsize 1}}
\put(5,0){\circle*{1,2}}\put(5,0){\makebox(0,-6)[c]{\scriptsize 2}}
\put(10,0){\circle*{1,2}}\put(10,0){\makebox(0,-6)[c]{\scriptsize
3}}
\put(15,0){\circle*{1,2}}\put(15,0){\makebox(0,-6)[c]{\scriptsize
4}}
\put(20,0){\circle*{1,2}}\put(20,0){\makebox(0,-6)[c]{\scriptsize
5}}
\put(25,0){\circle*{1,2}}\put(25,0){\makebox(0,-6)[c]{\scriptsize
6}}
\put(30,0){\circle*{1,2}}\put(30,0){\makebox(0,-6)[c]{\scriptsize
7}}
\put(35,0){\circle*{1,2}}\put(35,0){\makebox(0,-6)[c]{\scriptsize
8}}
\put(40,0){\circle*{1,2}}\put(40,0){\makebox(0,-6)[c]{\scriptsize
9}}
\put(45,0){\circle*{1,2}}\put(45,0){\makebox(0,-6)[c]{\scriptsize
10}}
\put(50,0){\circle*{1,2}}\put(50,0){\makebox(0,-6)[c]{\scriptsize
11}}
\put(55,0){\circle*{1,2}}\put(55,0){\makebox(0,-6)[c]{\scriptsize
12}}
%%%%%%%%%%%%%%%%%%%%%%%% Arcs %%%%%%%%%%%%%%%%%%%%%%%%%%%%%%%%%%%%%
\qbezier(0, 0)(22.5, 20)(45, 0) \qbezier(5, 0)(7.5, 7)(10, 0)
\qbezier(10,0)(20,13)(30,0) \qbezier(30,0)(35,10)(40,0)
\qbezier(20,0)(22.5,7)(25,0) \qbezier(25,0)(40,20)(55,0)
\qbezier(35,0)(42.5,11)(50,0)
%%%%%%%%%%%%%%%%%%%%%%%% titre %%%%%%%%%%%%%%%%%%%%%%%%%%%%%%%%%%%%%
\put(26,-5){\makebox(0,-6)[c]{\scriptsize (b) linear
representation}}
\end{picture}
}
\end{center}
\vspace{1.0cm} \caption{{\small Representations of the partition
$\pi=1\,10/2\,3\,7\,9/4/5\,6\,12/8\,11$}}
\label{fig:representations}
\end{figure}
These representations suggest two natural combinatorial parameters
on set partitions: the number of pairs of  crossing chords (resp.,
arcs) in the circular (resp., linear) representation,
denoted~$\croc$ (resp.,~$\crol$). For instance, if $\pi$ is the set
partition represented in Figure~\ref{fig:representations}, then
$\croc(\pi)=9$ and $\crol(\pi)=4$. Throughout this paper, the set of
all partitions of~$[n]$ will be denoted by~$\Pi_n$ and we
let~$\Pi_{n}^k$ denote the set of all partitions of~$[n]$
into~$k$~blocks.

A set partition each block of which has exactly two elements is 
often called a complete matching and its circular representation 
is often called a chord diagram. There has been significant interest 
(e.g., see~\cite{Tou,Ri,FlN,JosRu} and the
references there) in studying the distribution of the parameter
$\croc$ and $\crol$ in complete matchings (note that $\croc=\crol$ on the set  of matchings). Notably, a
remarkable exact counting formula (often called the Touchard-Riordan
formula) in terms of the ballot numbers was implicit in the work of
Touchard~\cite{Tou} and made explicit later by Riordan~\cite{Ri},
and Flajolet and Noy~\cite{FlN} proved that the distribution of the
parameter $\croc$ in a random chord diagram consisting of $n$ chords
is asymptotically Gaussian as $n\to\infty$.

 The enumeration of (general) set partitions by the parameter $\crol$ has also received
considerable interest (e.g., see~\cite{Bi,JosRu,KaZe,KaStZe}). This
is partly due to the fact that this parameter arises in the
combinatorial theory of continued fractions and of a natural
$q$-analog of Charlier polynomials~\cite{Bi,JosRu,KaStZe}. Moreover,
there exist combinatorial parameters on set partitions which have
the same distribution as the parameter $\crol$ over each~$\Pi_n^k$.
This is the case for the number of nestings of two arcs~\cite{KaZe}
and the major index for set partitions introduced in~\cite{ChenGe}.
 A classical result is that the number of noncrossing partitions (partitions $\pi$ with $\crol(\pi)=0$, or
equivalently, with $\croc(\pi)=0$) of $[n]$ is the Catalan number
$C_n=\frac{1}{n+1}{2n\choose n}$.
 Consider the enumerating polynomial $T_{n,k}(q)$ defined by
\begin{equation}\label{eq:def Tnk}
 T_{n,k}(q)=\sum_{\pi\in\Pi_n^k}q^{\crol(\pi)}.
\end{equation}
Biane~\cite{Bi} proved the Jacobi type continued fraction expansion
\begin{align*}
{\sum_{n\geq k\geq 0} T_{n,k}(q)\, a^k
t^n}&=\frac{1\;|}{|1-(a+[0]_q)z}-\frac{ a [1]_q
z^2\;|}{|1-(a+[1]_q)z}-\frac{a [2]_q z^2\;|}{|1-(a+[2]_q)z}
 -\cdots
\end{align*}
Stanton, Zeng and the author~\cite{KaStZe} obtained the expansion
(see~Equation~28 in~\cite{KaStZe})
\begin{align}\label{eq:gf_crol}
\sum_{n\geq k\geq 0} T_{n,k}(q)\,a^k\,t^n&=\sum_{k=0}^\infty
\frac{(aqt)^k}{\prod_{i=1}^k(q^i-q^i[i]_qt+a(1-q)[i]_qt)},
\end{align}
from which they derived the remarkable formula (see~Equation~32 in~\cite{KaStZe})
\begin{align}
\begin{split}\label{eq:distcr}
T_{n,k}(q)&=\sum_{j=1}^k (-1)^{k-j} \frac{[j]_q^n}{[j]_q!} q^{-kj}\, B_{n,k,j}(q),\quad\text{with}\\
&\qquad B_{n,k,j}(q)=\sum_{i=0}^{k-j} \frac{(1-q)^{i}}{[k-j-i]_{q}!}
 q^{\binom{k-j-i+1}{2}}\biggl( \binom{n}{i} q^{j}+\binom{n}{i-1}\biggr).
\end{split}
\end{align}
Here, as usual, $[a]_q!=[1]_q[2]_q\cdots[a]_q$ with
$[a]_q=\frac{1-q^a}{1-q}=1+q+\cdots+q^{a-1}$. Another remarkable
formula (see~\eqref{eq:JosRub}) for $T_{n,k}(q)$ was also recently established
by Josuat-Verg\`es and Rubey~\cite{JosRu}:
\begin{align}
\begin{split}\label{eq:JosRub} 
 T_{n,k}(q)&= \frac{1}{(1-q)^{n-k}}\sum_{j=0}^{k}\sum_{i=j}^{n-k} A_{i,j}(q) ,\quad\text{where}\\
A_{i,j}(q)&=(-1)^{i}\left(
{n\choose k+i}  {n\choose k-j} - {n\choose k+i+1}  {n\choose
k-j-1}\right){i\brack j}_q q^{{j+1 \choose 2}} 
\end{split}
\end{align}
and, as usual, the Gaussian coefficient ${i\brack j}_q$ is given by ${i\brack j}_q=\frac{[i]_q!}{[j]_q![i-j]_q!}$.

Note that if
seems extremely hard to get a convenient expression for the
distribution of $\crol$ (i.e., for any $j$, formulas  that give
the numbers of partitions of $[n]$ satisfying $\crol(\pi)=j$)
from the above formulas. Furthermore,  except for the formula for
the number of noncrossing partitions, it seems that nothing else is
known about the ``exact'' enumeration of set partitions by the
parameter $\croc$. In such a situation, it is customary to look for
asymptotic approximations (see e.g.~\cite[Chapter 4]{Sac} for an accurate study 
of the blocks of a random set partition or more recently~\cite{Knopf} for the study of records). 
A first step  was taken in~\cite{KaMean} 
where the author has obtained exact and asymptotic formulas for the
average values of the parameters $\crol$ and $\croc$ in a random set
partition. The present paper is mainly devoted to providing further
properties of the asymptotic behavior of these parameters. Our
approach is straightforward (although sometimes tedious):  most of
the results presented in this paper are derived from the generating
functions~\eqref{eq:gf_crol}-\eqref{eq:distcr} by using standard
methods of analytic combinatorics. We will also determine the maximum values 
of the parameter $\crol$ and $\croc$  on $\Pi_n$ and $\Pi_n^k$.

Before stating our results we need to set up some notation. We let
$X_{n}$ and $X_{n,k}$ (resp., $Y_n$ and $Y_{n,k}$) denote the random
variables equal to the value of  $\crol$ (resp., $\croc$)  taken,
respectively, over $\Pi_n$  and $\Pi_{n,k}$ endowed with the uniform
probability distribution, i.e., for any nonnegative integer $t$,
\begin{align*}
Pr\left( X_{n}=t \right) &=\frac{\# \{\pi\in\Pi_n\,:\,\crol(\pi)=t\}  }{B_n},\\
Pr\left( X_{n,k}=t \right)& =\frac{\#
\{\pi\in\Pi_n^k\,:\,\crol(\pi)=t\}  }{S_{n,k}},
\end{align*}
with similar statements for $Y_n$ and $Y_{n,k}$. Here $B_n$, the
$n$-th \emph{Bell number}, is the cardinality of~$\Pi_n$,
and~$S_{n,k}$, the $(n,k)$-th \emph{Stirling number of the second
kind}, is the cardinality of~$\Pi_n^k$.  As usual, the mean and the
variance of a random variable~$Z$ will be denoted by $\E(Z)$ and
$\Var(Z)$, and we use $\overset{d}{\longrightarrow}$ (resp.,
$\overset{p}{\longrightarrow}$) for convergence in distribution
(resp., in probability). We recall that a sequence of random
variables $\left(Z_{n}\right)_{n\geq 1}$ is said to be
\textit{asymptotically Gaussian} if
$\left(Z_n-\E(Z_n)\right)/\sqrt{\Var(Z_n)}\overset{d}{\longrightarrow}
\N(0, 1)$ as $n\to\infty$, where $\N(0,1)$ is the standard normal
distribution. In the rest of this paper, we take as granted the
elementary asymptotic calculus with usual Landau notations.

\subsection{Main results}

Let us first recall the closed form expressions for $\E(X_{n,k})$
and $\E(Y_{n,k})$ recently obtained  by the author.
\begin{thm}(Kasraoui~\cite{KaMean})\label{thm:exactmean-crolcroc}
For all integers $n\geq k\geq 1$,  we have
\begin{align}
\E(X_{n,k})&=\frac{1}{2}n(k-1)-\frac{5}{4}k(k-1)+\frac{3}{2}(n+1-k)\frac{S_{n,k-1}}{S_{n,k}},\label{eq:meanblock-crol}\\
\E(Y_{n,k})&=\frac{1}{2}n(k-1)-\frac{1}{2}n(4n-5k+1)\frac{S_{n-1,k-1}}{S_{n,k}}
-10\binom{n}{2}\frac{S_{n-2,k-2}}{S_{n,k}}\label{eq:meanblock-croc}\\
            &\quad +\binom{n}{4}\frac{S_{n-4,k-2}}{S_{n,k}}.\nonumber
\end{align}
\end{thm}
Using the well-known summation formula
$S_{n,k}=\frac{1}{k!}\sum_{j=0}^k(-1)^{j}{k\choose j} (k-j)^n$, it
is easy to see that
\begin{align}\label{eq:asymptotic-stirling}
S_{n,k}&=\frac{k^n}{k!}\left(1+o(1)\right) \quad\text{as
$n\rightarrow \infty$ uniformly for $k\leq n/(2\,\log n)$}.
\end{align}
Inserting~\eqref{eq:asymptotic-stirling} in
Theorem~\ref{thm:exactmean-crolcroc}, we immediately obtain the
following result.
 \begin{cor}\label{cor:mean-crolcroc-asympt}
For $1\leq k\leq  n/(2\,\log n)$ and as $n\to\infty$,  we have
\begin{align}
\E(X_{n,k})&=\frac{k-1}{2}\,n-\frac{5}{4}k(k-1)+o\left(1\right),\label{eq:asymptotic E(Xnk)}\\
\E(Y_{n,k})&=\frac{k-1}{2}\,n+o\left(1\right).\label{eq:asymptotic
E(Ynk)}
\end{align}
\end{cor}
The above result was explicitly stated for fixed integers $k$
in~\cite{KaMean}. We now present our first non-trivial result which
provides a useful expression for the variance of $X_{n,k}$. Its
proof is given in Section~2.

\begin{thm}\label{thm:ExactMomentsDistcrol_blocs}
For all integers $n\geq k\geq 1$, the variance of $X_{n,k}$
satisfies
\begin{align*}
\Var(X_{n,k})&=\frac{1}{12}(k^2-1) n -\frac{1}{72} k(k-1)(2 k+5)\\
 &\,+\frac{1}{12}\left(-15 n^2 +(56 k-43) n -2(k-1)(14k-1)\right)\frac{S_{n,k-1}}{S_{n,k}}\nonumber\\
 &\,+\frac{1}{12} (n-k+2)(27 n-27 k+1)\frac{S_{n,k-2}}{S_{n,k}}
  -\frac{9}{4}(n+1-k)^2\left(\frac{S_{n,k-1}}{S_{n,k}}\right)^2.\nonumber
\end{align*}
\end{thm}
The above theorem will be essentially derived by means of standard
but very tedious computations and manipulations of the generating
function expansion~\eqref{eq:gf_crol}. Combining the previous
theorem with~\eqref{eq:asymptotic-stirling}, we arrive at the
following result.

\begin{cor}\label{cor:varianceblock-crol-ktoinfty}
For $1\leq k\leq  n/(2\,\log n)$ and as $n\to\infty$, we have
\begin{align}\label{eq:asymptotic Var(Xnk)}
\Var(X_{n,k})=\frac{k^2-1}{12}\,n-\frac{1}{72} k(k-1)(2 k+5)
+o\left(1\right).
\end{align}
\end{cor}

The two following theorems are probably the most interesting 
results of the present paper. In Section~3, using expression~\eqref{eq:distcr}  and
Curtiss' theorem  for sequences of moment generating functions, we
prove a central limit law for the random variable $X_{n,k}$.
\begin{thm}\label{thm:LimitDistcrol-ktoinfty}
For $2\leq k\leq n/(2\,\log n)$, the distribution of $X_{n,k}$ is
asymptotically Gaussian as $n\to\infty$.
\end{thm}
 Let us mention that it seems that the asymptotic distribution of~$X_{n,k}$ for
other types of interesting ranges for $k$ (for instance, $k\sim
a\,n$ with $0<a<1$) can be analyzed with a similar method. However,
the asymptotic analysis of the moment generating function
of~$X_{n,k}$ (see Section~3) seems much more difficult for $k\sim
a\,n$ than for $k\leq n/(2\,\log n)$. This explains why, in this
paper, we have limited our study to the latter range.

We now turn our attention to the random variable $Y_{n,k}$. In
contrast to the parameter~$\crol$, we have no convenient expression
for the distribution of~$\croc$. Nonetheless, exploiting the
combinatorial ``closeness'' of the parameters~$\croc$ and~$\crol$
(see Lemma~\ref{lem:croc-crol} for a precise statement) and the
results on the distribution of $X_{n,k}$ expounded above, we have
managed to obtain interesting (but, necessarily, weaker) results on the asymptotic
distribution of $Y_{n,k}$. The following result is demonstrated in
Section~5.1.
\begin{thm}\label{thm:LimitDistcroc-ktoinfty}
 For $k=o\left(\sqrt{n}\right)$ and as $n\to\infty$,
\begin{itemize}
\item[(i)] the variance of the distribution of $Y_{n,k}$ satisfies
\begin{align}\label{eq:asymptotic Var(Ynk)}
\Var(Y_{n,k})=\frac{k^2-1}{12}\, n+ O(k^3\,\sqrt{n}),
\end{align}
\item[(ii)] the distribution of $Y_{n,k}$ is asymptotically Gaussian.
\end{itemize}
\end{thm}

It seems more difficult to deal with the asymptotic distribution
of~$X_n$ and~$Y_n$. Actually, numerical evidence
(see~Figure~\ref{fig:histo_crol}) naturally lead to the following conjecture.
\begin{conj}
The distributions of $X_n$ and $Y_n$ are asymptotically Gaussian.
\end{conj}

\begin{figure}[h]
\centerline{
\includegraphics[width=5cm,height=8cm]{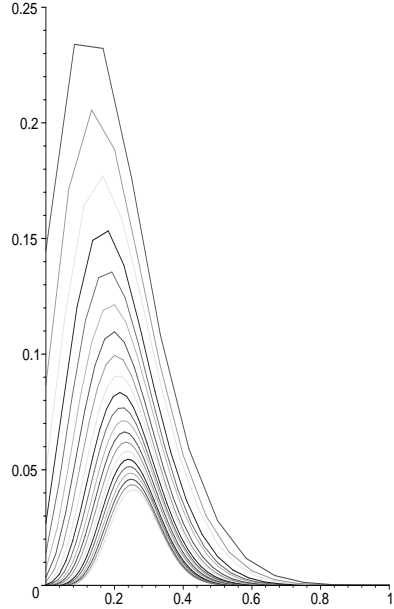}}
\caption{Plots of the distributions of $\crol$ on $\Pi_n$ for $n =
10, \ldots, 30$. The horizontal axis is normalized (by a factor of
$1/a_n$ with $a_n=\left\lfloor\frac{1}{3}{n-1\choose
2}\right\rfloor$) and rescaled to 1, so that the curves display
$Pr\left(\frac{X_n}{a_n}= x\right)$, for $ x = 0, 1/a_n , 2/a_n ,
\ldots$}\label{fig:histo_crol}
\end{figure}
We have not been able to prove this conjecture. We can, however,
show that the distributions of~$X_n$ and~$Y_n$ are concentrated
around their mean. We first recall the approximation of~$\E(X_{n})$
and~$\E(Y_{n})$ recently obtained by the author.
 \begin{thm}(Kasraoui~\cite{KaMean})\label{thm:mean-crolcroc1}
 As $n\rightarrow\infty$, the means $\E(X_{n})$ and $\E(Y_{n})$ are both equal to
\begin{align*}
\frac{n^2}{2\log n}\left( 1+\frac{\log\log n}{\log
n}\left(1+o\left(1\right)\right) \right).
\end{align*}
\end{thm}
 In Section~4,  we deduce from the
expression~\eqref{eq:meanblock-crol} for $\E(X_{n,k})$ and the
formula of $\Var(X_{n,k})$ in
Theorem~\ref{thm:ExactMomentsDistcrol_blocs} a useful expression for
$\Var(X_{n})$ from which we will obtain the following fairly good
asymptotic approximation of~$\Var(X_{n})$.
\begin{thm}\label{thm:MomentsDistcrol}
 The variance of $X_{n}$ satisfies, as $n\rightarrow\infty$,
\begin{align}\label{eq:MomentsDistcrol}
\Var(X_n)&=\frac{n^{3}}{3\left(\log n\right)^{2}} \left
(1+2\frac{\log\log n}{\log n} (1+o(1)) \right).
\end{align}
\end{thm}
Although the proximity of the parameters $\croc$ and~$\crol$ lead us
to believe that $\Var(Y_n)$ and $\Var(X_n)$ are asymptotically
equivalent, we are currently unable to obtain a satisfactory
approximation of $\Var(Y_n)$. In Section~5.2,  we derive
only the following  bound (which seems to be far from sharp  but is
sufficient for our purpose) by exploiting the closeness of~$\croc$
to~$\crol$.
\begin{thm}\label{thm:MomentsDistcroc}
 The variance of $Y_{n}$ satisfies, as $n\rightarrow\infty$,
\begin{align}\label{eq:MomentsDistcroc}
\Var(Y_n)&=O\left(\frac{n^{4}}{\left(\log n\right)^{4}}\right).
\end{align}
\end{thm}

Combining Theorem~\ref{thm:mean-crolcroc1} with
equations~\eqref{eq:MomentsDistcrol}--\eqref{eq:MomentsDistcroc}, we
see  that\linebreak ${\sqrt{\Var(X_n)}/\E(X_n)\to 0}$ and
$\sqrt{\Var(Y_n)}/\E(Y_n)\to 0$ as $n\to\infty$.  This leads
immediately (by Chebyshev's inequality) to the following result.
\begin{cor}
 As $n\to\infty$, the
distributions of $X_n$ and $Y_n$ are concentrated around their mean,
i.e., $\frac{X_n}{\E(X_n)}\overset{p}{\longrightarrow} 1$ and
$\frac{Y_n}{\E(Y_n)}\overset{p}{\longrightarrow} 1$.
\end{cor}

 Last but not least, in Section~6, we find maximum values of the
parameters $\crol$ and~$\croc$ on $\Pi_n$ and $\Pi_n^k$. This answers a question 
of P. Nadeau~\cite{Nadeau}  which is important for the
comprehension of the distribution of the parameters $\crol$ and
$\croc$. The answer is, in full generality, far from obvious.
Note that it is easy to see that for each $n\geq k\geq 1$, there is a partition $\pi$ in $\Pi_{n,k}$ 
such that $\crol(\pi)=\croc(\pi)=0$ (there are in fact $N(n,k)$ such partitions where $N(n,k)$
 is a Narayana number).
 
\begin{thm}\label{thm:max-cro-bl}
 Let $M^{(\ell)}_{n,k}$ (resp., $M^{(c)}_{n,k}$) denote the maximum value
of~$\crol(\pi)$ (resp., $\croc(\pi)$) taken over all
$\pi\in\Pi_n^k$. Then,  we have
\begin{enumerate}
 \item \begin{enumerate}
        \item if $1\leq k\leq\left\lfloor \tfrac{n}{2}\right\rfloor $,
$M^{(\ell)}_{n,k}=(k-1)n -3 \binom{k}{2}$,
         \item if $\left\lceil \tfrac{n}{2}\right\rceil \leq  k\leq n$,
$M^{(\ell)}_{n,k}=\binom{n-k}{2}$;
       \end{enumerate}
\item \begin{enumerate}
        \item if $k\leq \tfrac{n}{3}$,
$M^{(c)}_{n,k}=2\binom{k}{2}\left\lfloor
\tfrac{n}{k}\right\rfloor+2\binom{r}{2}=(k-1)n -r(k-r)$, where $r$ is the remainder in the division of $n$ by $k$,
         \item if $\tfrac{n}{3}\leq k<\tfrac{n}{2}$ or $\tfrac{n}{2}\leq k\leq
n-6$, $M^{(c)}_{n,k}=6\binom{\lfloor
\tfrac{a}{2}\rfloor}{2}+2\lfloor \tfrac{a}{2}\rfloor \chi\left(
a\equiv 1\pmod{2}\right)$
where we set $a=n-k$,
 \item if $k\geq n-5$ and $k\geq \tfrac{n}{2}$,
$M^{(c)}_{n,k}=\binom{n-k}{2}$.
       \end{enumerate}
\end{enumerate}
\end{thm}

The proof of the above result is far from trivial and relies on a tedious discrete optimization. 
It is also worth noting that the first part of Theorem~\ref{thm:max-cro-bl} could be also derived 
by computing the degree of~$T_{n,k}(q)$ in $q$ thanks to the expressions~\eqref{eq:distcr} or~\eqref{eq:JosRub}, but this method
is inefficient for determining $M^{(c)}_{n,k}$. \newline
 \indent 
In conjunction with the above result, it is not difficult to determine the global maximas of the functions $k\mapsto M^{(\ell)}_{n,k}$ 
and $k\mapsto M^{(c)}_{n,k}$ defined on $[n]$. This yields the following result.

\begin{thm}\label{thm:max-cro}
 Let $M^{(\ell)}_{n}$ (resp., $M^{(c)}_{n}$) denote the maximum value
of~$\crol(\pi)$ (resp., $\croc(\pi)$) taken over all~${\pi\in
\Pi_n}$. Then, we have
\begin{enumerate}
 \item if $n\geq 1$, $ M^{(\ell)}_{n}=\left\lfloor
\tfrac{1}{3}{n-1\choose 2}\right\rfloor $,
\item \begin{enumerate}
        \item if $n\geq 5$ and $n\equiv 0\pmod{3}$, $M^{(c)}_{n}=\left\lfloor
\tfrac{2}{3}{n-1\choose 2}\right\rfloor$,
         \item if $n\geq 5$ and $n\equiv 1,2\pmod{3}$,
$M^{(c)}_{n}=\left\lfloor \tfrac{2}{3}{n-2\choose 2}\right\rfloor$.
       \end{enumerate}
\end{enumerate}
\end{thm}

%\newpage
%%%%%%%%%%%%%%%%%%%%%%%%%%%%%%%%%%%%%%%%%%%%%%%%%%%%%%%%%%%%%%%%%%%%%%%%%%%%%%%%%%%%%%%%%%%%%%%%%%%%%%%%%%%%%%%%%%%%%%%%%%%%%%%%%
%%%%%%%%%%%%%%%%%%%%%%%%%%%%%%%%%%%%%%%%%%%%%%%%%%%%%%%%%%%%%%%%%%%%%%%%%%%%%%%%%%%%%%%%%%%%%%%%%%%%%%%%%%%%%%%%%%%%%%%%%%%%%%%%%
%%%%%%%%%%%
%%%%%%%%%%%                 The variance of X_{n,k}
%%%%%%%%%%%
%%%%%%%%%%%%%%%%%%%%%%%%%%%%%%%%%%%%%%%%%%%%%%%%%%%%%%%%%%%%%%%%%%%%%%%%%%%%%%%%%%%%%%%%%%%%%%%%%%%%%%%%%%%%%%%%%%%%%%%%%%%%%%%%%
%%%%%%%%%%%%%%%%%%%%%%%%%%%%%%%%%%%%%%%%%%%%%%%%%%%%%%%%%%%%%%%%%%%%%%%%%%%%%%%%%%%%%%%%%%%%%%%%%%%%%%%%%%%%%%%%%%%%%%%%%%%%%%%%%
\section{The variance of $X_{n,k}$: proof of Theorem~\ref{thm:ExactMomentsDistcrol_blocs}}

%%%%%%%%%%%%%%%%%%%%%%%%%%%%%%%%%%%%
% new subsection: Preuve variance CROISEMENT
%%%%%%%%%%%%%%%%%%%%%%%%%%%%%%%%%%%%
 Our proof of Theorem~\ref{thm:ExactMomentsDistcrol_blocs} essentially relies on the
generating function expansion~\eqref{eq:gf_crol}. Recall that a
useful property of the probability generating function $G(q)$ of a
non-negative integer valued random variable~$Z$ is that its $m$-th
derivative at $q=1$ gives the $m$-th factorial moment of $Z$. By
definition, the probability generating function of~$X_{n,k}$ is
$p_{n,k}(q)=T_{n,k}(q)/S_{n,k}$, where $T_{n,k}(q)$ is defined
in~\eqref{eq:def Tnk}. Consequently, we have
$\E\big(X_{n,k}(X_{n,k}-1)\big)=T_{n,k}^{\prime\prime}(1)/S_{n,k}$,
whence
\begin{align}\label{eq:Var from gf}
\Var(X_{n,k})=\frac{T_{n,k}^{\prime\prime}(1)}{S_{n,k}}+\E(X_{n,k})-\E(X_{n,k})^2.
\end{align}
To find a ``convenient'' expression for $\Var(X_{n,k})$, we just
have to find a formula for $T_{n,k}^{\prime\prime}(1)$
since~\eqref{eq:meanblock-crol} already provides a formula
for~$\E(X_{n,k})$. We will ``extract'' a formula
for~$T_{n,k}^{\prime\prime}(1)$ from~\eqref{eq:gf_crol} by a routine
but unpleasant computation.
 \begin{prop}\label{prop:2nd derivative} For all integers $n\geq k\geq 1$, we have
 \begin{align*}
\frac{T_{n,k}''(1)}{S_{n,k}}&=\frac{(k-1)^2}{4}n^2  -\frac{k-1}{12}(15 k^2-16 k+5) n+\frac{k (k-1)}{144}(225 k^2-229 k\\
&\;+170)+\frac{1}{12}\big( (18 k -33)n^2 -(63 k^2-137 k+79) n +(k-1)(45 k^2-73 k\\
&\;+20)\big) \frac{S_{n,k-1}}{S_{n,k}}+\frac{1}{12} \big(
27n^2-(54k-55)n  +(k-2)(27 k-1)  \big)  \frac{S_{n,k-2}}{S_{n,k}}.
\end{align*}
\end{prop}

 Inserting Proposition~\ref{prop:2nd derivative} and~\eqref{eq:meanblock-crol}
in \eqref{eq:Var from gf} gives Theorem~\ref{thm:ExactMomentsDistcrol_blocs}. 
So, to complete the proof of Theorem~\ref{thm:ExactMomentsDistcrol_blocs}, 
it suffices to verify Proposition~\ref{prop:2nd derivative}.

\emph{Proof of Proposition~\ref{prop:2nd derivative}}
By~\eqref{eq:gf_crol}, the ordinary generating function  of the
$T_{n,k}^{\prime\prime}(1)$'s satisfies
\begin{align}\label{eq:doublegenfunc-crol}
\sum_{n\geq k\geq0}T_{n,k}^{\prime\prime}(1)\,a^kt^n
=\sum_{k\geq0}(at)^kF_k^{\prime\prime}(1),
\end{align}
where $F_k(q)= {\prod_{i=1}^kf_i(q)}$ with
\begin{align} \label{eq:definition_fi}
f_i(q)=f_i(q;a,t)=\frac{q}{q^i-q^i[i]_qt+a(1-q)[i]_qt}.
\end{align}
Using Leibnitz's rule for the derivative of a product, we get
\begin{align}
F_k^{\prime\prime}(1)&= F_k(1)\left(
 \left(\sum_{i=1}^k \frac{f_i'(1)}{f_i(1)}\right)^2
 + \sum_{i=1}^k \left(\frac{f_i^{\prime\prime}(1)}{f_i(1)}-\left(\frac{f_i^{\prime}(1)}{f_i(1)}\right)^2\right)
 \right).\label{eq:derive2nde}
\end{align}
Moreover, using expression \eqref{eq:definition_fi} for $f_i(q)$, after a
routine computation followed by partial fraction decompositions, we
arrive at
\begin{align}
&\frac{f_i'(1)}{f_i(1)}=-\frac{3 i}{2}-\frac{1-3 t+2 a t}{2
t}+\frac{-1+t-2 a t}{2 t (-1+i t)},\label{eq:f'surf}\\
&\frac{f_i^{\prime\prime}(1)}{f_i(1)}-\left(\frac{f_i^{\prime}(1)}{f_i(1)}\right)^2\label{eq:f"surf}\\
&\quad=-\frac{i^2}{12}+\frac{i (1+9 t+12 a t)}{6 t}+\frac{5+36 a t-17 t^2+12 a^2 t^2}{12 t^2}\nonumber\\
&\quad\quad+\frac{4-3 t+24 a t-t^2-6 a t^2+12 a^2 t^2}{6 t^2 (-1+i
t)} +\frac{1-2 t+4 a t+t^2-4 a t^2+4 a^2 t^2}{4 t^2
(-1+it)^2}.\nonumber
\end{align}
For  $k\geq0$, define  power series $U_k(t)$ and $V_k(t)$  by
\begin{align}\label{eq:Def_Sk&Tk}
U_k(t):= \sum_{i=1}^k\frac{1}{1-it}\quad\text{and}\quad  V_k(t):=
\sum_{i=1}^k\frac{1}{(1-it)^2},
\end{align}
and set $G_k(t):=t^k F_k(1)$. By
specializing~\eqref{eq:definition_fi} at $q=1$, we have
\begin{align}\label{eq:ordFGstirling}
G_k(t)=t^k
{\prod_{i=1}^kf_i(1)}=\frac{t^k}{\prod_{i=1}^k(1-it)}=\sum_{n\geq0}S_{n,k}\,t^n,
\end{align}
where the last equality is a well-known power series expansion.

Combining~\eqref{eq:derive2nde}
with~\eqref{eq:f'surf}--\eqref{eq:ordFGstirling}, it is easy (but
unpleasant except if we use a computer algebra system) to show that
\begin{align*}
t^k F_k^{\prime\prime}(1)&=
G_k(t)\left(p_k(a,t)+q_k(a,t)\,U_k(t)+r_k(a,t)\,\left(U_k(t)\right)^2+s_k(a,t)\,V_k(t)\right),
\end{align*}
where  $p_k(a,t)$, $q_k(a,t)$, $r_k(a,t)$ and $s_k(a,t)$ are
polynomials in ${\mathbb C}[a,a^{-1},t,t^{-1}]$ given by
\begin{align}
 p_k(a,t)&= \frac{k}{144}  \left(-98+183 k-166 k^2+81 k^3+144 a^2 (1+k)+72 a \left(2-k+3 k^2\right)\right)\nonumber \\
    &\quad+\frac{k \left(1-8 k+9 k^2+12 a (3+k)\right)}{12 t}+\frac{k (5+3 k)}{12 t^2},\qquad\qquad\nonumber\\
 q_k(a,t)&= \frac{1}{12} \left(2-9 k+9 k^2-24 a^2 (1+k)+6 a \left(2+5 k-3 k^2\right)\right)\\
    &\quad+\frac{2+5 k-3 k^2-8 a (2+k)}{4 t}+\frac{-4-3 k}{6 t^2},\nonumber\\
 r_k(a,t)&= s_k(a,t)=\frac{1}{4}-a+a^2+\frac{-1+2a}{2t}+\frac{1}{4 t^2}.\nonumber
\end{align}
In conjunction with~\eqref{eq:doublegenfunc-crol}, this implies that
\begin{align}
\begin{split}\label{eq:T''(1) commecoeff}
T_{n,k}''(1)&=[a^{k}t^{n}]\sum_{k\geq 0} a^kG_k(t)p_k(a,t)+[a^{k}t^{n}]\sum_{k\geq 0} a^k G_k(t)U_k(t)q_k(a,t)\\
&\quad+ [a^{k}t^{n}]\sum_{k\geq 0}
a^kG_k(t)\left(\left(U_k(t)\right)^2+V_k(t)\right)r_k(a,t).
\end{split}
\end{align}
Here, as usual,  $[a^{i}t^{j}] W(a,t)$ is for the coefficient of $a^i
t^j$ in the power series expansion of~$W(a,t)$.
\begin{lem}\label{lem:FGstirling1}
We have the formal power series expansions
\begin{align*}
C_1(a,t)&:=\sum_{k\geq 0} a^k G_k(t)
          =\sum_{n,k\geq0} S_{n,k} \,a^k t^n,\\
C_2(a,t)&:=\sum_{k\geq 0} \,a^k G_k(t) U_k(t)
          =\sum_{n,k\geq0} n S_{n,k}\,a^k t^n,\\
C_3(a,t)&:=\sum_{k\geq 0} a^k G_k(t) \left( \left(U_k(t)\right)^{2}
+ V_k(t)\right)
          =\sum_{n,k\geq0} n(n+1)S_{n,k} \,a^kt^n,%\label{eq:FGstirling3}
\end{align*}
where $U_k(t)$, $V_k(t)$ and $G_k(t)$ are defined
in~\eqref{eq:Def_Sk&Tk} and~\eqref{eq:ordFGstirling}.
\end{lem}
\begin{proof}
 The first expansion is immediate from~\eqref{eq:ordFGstirling}.
If we derive  twice each side of~\eqref{eq:ordFGstirling}, we obtain
the power series expansions
\begin{align*}
 G_k'(t) &=\frac{1}{t}G_k(t)U_k(t)=\frac{1}{t} \sum_{n\geq0} nS_{n,k}\, t^n,\\
 G_k^{\prime\prime}(t)&=\frac{1}{t^{2}}G_k(t)\left(\left(U_k(t)\right)^{2}+V_k(t)-2U_k(t)\right)
           =\frac{1}{t^{2}}\sum_{n\geq0} n(n-1)S_{n,k} \,t^n,
\end{align*}
from which it is straightforward to deduce the expansions of
$C_2(a,t)$ and $C_3(a,t)$.
\end{proof}

It is now a routine matter to derive a formula for
$T_{n,k}^{\prime\prime}(1)$ (which involves only polynomials in~$n$
and~$k$ and Stirling numbers) from~\eqref{eq:T''(1) commecoeff}.
Indeed, after routine coefficient extractions in~\eqref{eq:T''(1)
commecoeff} based on Lemma~\ref{lem:FGstirling1}, it is easy to
obtain
\begin{align*}
T_{n,k}''(1) &=
 (k-1)(k-2)\,S_{n,k-2}-2 (k-1) \,n\,S_{n,k-2}+  n (n+1) \,S_{n,k-2}  \\
&
 \;+\frac{1}{2}(k-1)(3k^{2}-7k+6)\,S_{n,k-1}-\frac{1}{2}(k-3)(3k-2) \,n\, S_{n,k-1}\\
&
\;-n (n+1) \,S_{n,k-1}+\frac{1}{144}k(k-1)(81k^{2}-85k+98)\, S_{n,k}\\
&
\;+\frac{1}{12}(3k-1)(3k-2)\,n \,S_{n,k}+\frac{1}{4}n (n+1)\,S_{n,k}+(k+1)(k+2)\,S_{n+1,k-1}   \\
&
\;-2 (k+1)  \,(n+1)\,S_{n+1,k-1}+(n+1) (n+2)\,S_{n+1,k-1}\\
&
\;+\frac{1}{12}k(9k^{2}-8k+1)\,S_{n+1,k}-\frac{1}{4}(3k+1)(k-2) \,(n+1)\,S_{n+1,k}\\
&
\;-\frac{1}{2}(n+1) (n+2)\, S_{n+1,k}+\frac{1}{12}k(3k+5)\,S_{n+2,k}\\
& \; -\frac{1}{6}(3k+4)
\,(n+2)\,S_{n+2,k}+\frac{1}{4}(n+2)(n+3)\,S_{n+2,k}.
\end{align*}
By replacing in the latter equation each occurrence of the left hand
sides of the three following identities
\begin{align*}
&S_{n+1,k}=S_{n,k-1}+k S_{n,k},\quad
S_{n+1,k-1}=S_{n,k-2}+ (k-1) S_{n,k-1},\\
&S_{n+2,k}=S_{n,k-2}+(2k-1) S_{n,k-1}+k^{2}S_{n,k},
\end{align*}
by the corresponding right hand sides, we arrive at
Proposition~\ref{prop:2nd derivative}. \qed

%\newpage
%%%%%%%%%%%%%%%%%%%%%%%%%%%%%%%%%%%%%%%%%%%%%%%%%%%%%%%%%%%%%%%%%%%%%%%%%%%%%%%%%%%%%%%%%%%%%%%%%%%%%%%%%%%%%%%%%%%%%%%%%%%%%%%%%
%%%%%%%%%%%%%%%%%%%%%%%%%%%%%%%%%%%%%%%%%%%%%%%%%%%%%%%%%%%%%%%%%%%%%%%%%%%%%%%%%%%%%%%%%%%%%%%%%%%%%%%%%%%%%%%%%%%%%%%%%%%%%%%%%
%%%%%%%%%%%
%%%%%%%%%%%                 Limit laws of  $\crol$ over partitions k t infty
%%%%%%%%%%%
%%%%%%%%%%%%%%%%%%%%%%%%%%%%%%%%%%%%%%%%%%%%%%%%%%%%%%%%%%%%%%%%%%%%%%%%%%%%%%%%%%%%%%%%%%%%%%%%%%%%%%%%%%%%%%%%%%%%%%%%%%%%%%%%%
%%%%%%%%%%%%%%%%%%%%%%%%%%%%%%%%%%%%%%%%%%%%%%%%%%%%%%%%%%%%%%%%%%%%%%%%%%%%%%%%%%%%%%%%%%%%%%%%%%%%%%%%%%%%%%%%%%%%%%%%%%%%%%%%%
\section{Limiting distribution of $X_{n,k}$}

This section is devoted to proving
Theorem~\ref{thm:LimitDistcrol-ktoinfty}. For simplicity, throughout
this section, \emph{all asymptotic are meant for $2\leq
k\leq n/(2\,\log n)$ and $n\to\infty$ unless otherwise stated} and
we denote $\mu_{n,k}=\E(X_{n,k})$ and
${\sigma_{n,k}}^2=\Var(X_{n,k})$.

Let $\tilde{M}_{n,k}(t)$ (resp., $\tilde{P}_{n,k}(q)$) be the moment
(resp., probability) generating function of the random
variable~$\tilde{X}_{n,k}=\left(X_{n,k}-\mu_{n,k}\right)/\sigma_{n,k}$.
Then, we have  $\tilde{M}_{n,k}(t)=\tilde{P}_{n,k}\left(e^t\right)$,
whence
\begin{align}\label{eq:MomGF-Xc}
\tilde{M}_{n,k}(t)=
\exp\left({-\frac{\mu_{n,k}}{\sigma_{n,k}}t}\right)\,
\frac{T_{n,k}\left(\exp(t/\sigma_{n,k}) \right)}{S_{n,k}},
\end{align}
where $T_{n,k}(q)$ is defined in~\eqref{eq:def Tnk}. Using
expression~\eqref{eq:distcr} and only elementary asymptotic
analysis, we shall prove that $\tilde{M}_{n,k}(t)$ converges
pointwise on $\mathbb{R}$ to the function $g(t):=\exp(t^2/2)$.  By a
celebrated theorem of Curtiss (see e.g.~Theorem~2.7 in~\cite{Sac}),
this will imply that $\tilde{X}_{n,k}\overset{d}{\longrightarrow}
\N(0,1)$, as stated in Theorem~\ref{thm:LimitDistcrol-ktoinfty}.

\begin{lem}\label{lem:asymptotic-MomGF-1}
 Let $t\in \mathbb{R}$ and $u:=\exp(t/\sigma_{n,k})$. 
 Then, for $2\leq k\leq  n/(2\,\log n)$ and as $n\to\infty$,
we have
\begin{align}
T_{n,k}(u)=u^{-k^2}\frac{{[k]_u} ^n}{[k]_u!}
\left(1+o\left(1\right)\right).
\end{align}
\end{lem}
Inserting Lemma~\ref{lem:asymptotic-MomGF-1} and the
approximation~\eqref{eq:asymptotic-stirling} in~\eqref{eq:MomGF-Xc},
we arrive at
\begin{align}\label{eq:asymptotic-MomGF-1}
\tilde{M}_{n,k}(t)= u^{-k^2-\mu_{n,k}}\left(\frac{[k]_u
}{k}\right)^n \frac{k!}{[k]_u!} \left(1+o(1)\right).
\end{align}

\begin{lem}\label{lem:asymptotic-MomGF-2}
 Let $t\in \mathbb{R}$ and $u:=\exp(t/\sigma_{n,k})$.
Then, for $2\leq k\leq  n/(2\,\log n)$ and as $n\to\infty$, we have
\begin{align*}
(a)&\quad \log \left(\frac{[k]_u}{k}\right)=\frac{k-1}{2\,\sigma_{n,k}}t+\frac{k^2}{24\,\sigma_{n,k}^2}t^2+o\left(\frac{k^2}{\sigma_{n,k}^2}\right),\\
(b)&\quad \log
\left(\frac{[k]_u!}{k!}\right)=\frac{k(k-1)}{4\,\sigma_{n,k}}t+o(1).
\end{align*}
\end{lem}
Inserting Lemma~\ref{lem:asymptotic-MomGF-2}
in~\eqref{eq:asymptotic-MomGF-1}, and then using the
approximations~\eqref{eq:asymptotic E(Xnk)} and~\eqref{eq:asymptotic
Var(Xnk)} for $\mu_{n,k}$ and  $\sigma_{n,k}^2$, after a routine
computation, we obtain
\begin{align*}
\log \tilde{M}_{n,k}(t)
&=\frac{-k^2-\mu_{n,k}}{\sigma_{n,k}}t+n\log \left(\frac{[k]_u}{k}\right)-\log \left(\frac{[k]_u!}{k!}\right)+o(1)\\
&=\left(-k^2-\mu_{n,k}+n\frac{k-1}{2}-\frac{k(k-1)}{4}\right)\frac{t}{\sigma_{n,k}}+\frac{n\,k^2}{24\sigma_{n,k}^2}t^2+o(1)\\
&=o(k)\cdot\frac{t}{\sigma_{n,k}}+\frac{t^2}{2}+o(1)=\frac{t^2}{2}+o(1),
\end{align*}
as desired. So, to complete the proof of
Theorem~\ref{thm:LimitDistcrol-ktoinfty}, it suffices to prove
Lemma~\ref{lem:asymptotic-MomGF-1} and
Lemma~\ref{lem:asymptotic-MomGF-2}.

%%%%%%%%%%%%%%%%%%%%%%%%%%%%%%%%%%%%%%%%%%%%%%%%%%%%%%%%
% \textit{Proof of Lemma~\ref{lem:asymptotic-MomGF-2}.}
%%%%%%%%%%%%%%%%%%%%%%%%%%%%%%%%%%%%%%%%%%%%%%%%%%%%%%%%

\paragraph{\textit{Proof of Lemma~\ref{lem:asymptotic-MomGF-2}}}
(a) By definition,
$[k]_{u}=\left(e^{tk/\sigma_{n,k}}-1\right)\left(e^{t/\sigma_{n,k}}-1\right)^{-1}$.
Using the asymptotic expansion $e^x=1+x+x^2/2+x^3/6+o(x^2)$ (valid
as $x\to 0$), it is easy to check that, for any (fixed) real~$t$, if
$w/v\to 0$ as $(v,w)\to (0,0)$, then we have, as $(v,w)\to (0,0)$,
\begin{align*}
\left(e^{v\,t}-1\right)\left(e^{w\,t}-1\right)^{-1}
&=v\,w^{-1}\left(1+(v-w)\,t/2+v^2\,t^2/6+o(v^2)\right).
\end{align*}
Noting that $k/\sigma_{n,k}$ is $O\left(n^{-1/2}\right)$
by~\eqref{eq:asymptotic Var(Xnk)} and specializing the above formula
at $v=k/\sigma_{n,k}$ and $w=1/\sigma_{n,k}$, we arrive at
\begin{align}\label{eq:asympt u-int}
[k]_{u} &=k
\left(1+\frac{k-1}{2\,\sigma_{n,k}}t+\frac{k^2}{6\,\sigma_{n,k}^2}t^2(1+o(1))\right).
\end{align}
This implies, by the expansion $\log (1+x)=x-x^2/2+o(x^2)$ as $x\to
0$, that
\begin{align*}
\log \left(\frac{[k]_{u}}{k}\right)
&=\frac{k-1}{2\,\sigma_{n,k}}t+\frac{k^2}{6\,\sigma_{n,k}^2}t^2-\frac{1}{2}\left(\frac{k-1}{2\,\sigma_{n,k}}t
+\frac{k^2}{6\,\sigma_{n,k}^2}t^2\right)^2+o\left(\frac{k^2}{\sigma_{n,k}^2}\right),
\end{align*}
which leads, after straightforward simplifications, to the desired result.\\

(b) Let $I_k(q):=\frac{[k]_q!}{k!}$. Sachkov (see
e.g.~\cite[Section~1.3.1, p.29]{Sac}) proved that
\begin{align}\label{eq:Mom-Mahonian}
 \log I_k(e^x)=\frac{k(k-1)}{4}x+\sum_{\ell=1}^\infty b_{2\ell}\frac{x^{2\ell}}{2\ell\,(2\ell)!}\sum_{j=1}^k (j^{2\ell}-1)\qquad(|x|<2\pi),
\end{align}
where $b_{2\ell}$ are the Bernoulli numbers.  Since for all
$\ell\geq 1$
$$\sigma_{n,k}^{-2\ell}\sum_{j=1}^k (j^{2\ell}-1)=\sigma_{n,k}^{-2\ell}\cdot O\left(k^{2\ell+1} \right)
=O\left(\frac{k}{n^{\ell}} \right)\quad \textrm{as $k\to\infty$},
$$ we have, in view of the expansion~\eqref{eq:Mom-Mahonian},
\begin{align*}
\log I_k(\exp(t/\sigma_{n,k}))-\frac{k(k-1)}{4\,\sigma_{n,k}}t =
\sum_{\ell=1}^\infty
b_{2\ell}\frac{t^{2\ell}}{2\ell\,(2\ell)!}\,\sigma_{n,k}^{-2\ell}\sum_{j=1}^k
(j^{2\ell}-1)\to 0.
\end{align*}
This is exactly the formula (b) in
Lemma~\ref{lem:asymptotic-MomGF-2}. \qed

%%%%%%%%%%%%%%%%%%%%%%%%%%%%%%%%%%%%%%%%%%%%%%%%%%%%%%%%
% \textit{Proof of Lemma~\ref{lem:asymptotic-MomGF-1}.}
%%%%%%%%%%%%%%%%%%%%%%%%%%%%%%%%%%%%%%%%%%%%%%%%%%%%%%%%

\paragraph{\textit{Proof of Lemma~\ref{lem:asymptotic-MomGF-1}}}
 \textit{Step 1}. Consider the $B_{n,k,j}(q)$'s defined in~\eqref{eq:distcr}. We claim that,
for  any positive real number $q$,  we have
\begin{align}\label{eq:lemma-MomGF-1bis}
\left|T_{n,k}(q)-B_{n,k,k}(q)\right|\leq \left(1+2 \,q^k \right)
\exp\left(\frac{n|1-q|}{q} \right)\sum_{j=1}^{k-1}
\frac{[j]_q^n}{[j]_q!} q^{-kj}
\end{align}
if $k\leq n/(2\,\log n)$ and $n$ is enough large. This follows
immediately from~\eqref{eq:distcr} and the relation (valid for
$k\leq n/(2\,\log n)$ and $n$ enough large)
\begin{align}\label{eq:lemma-MomGF-1}
\left|B_{n,k,j}(q)\right| &\leq \left(1+2 \,q^k \right)
\exp\left(\frac{n|1-q|}{q} \right) \quad(1\leq j\leq k-1,\; q>0).
\end{align}

To prove \eqref{eq:lemma-MomGF-1}, first observe that, for any real
$q>0$, we have
$$
\frac{q^{\binom{\ell+1}{2}}}{[\ell]_{q}!}=q^{\ell}
\prod_{i=1}^{\ell} \frac{q^{i-1}}{[i]_{q}} =q^{\ell}
\prod_{i=1}^{\ell} \frac{q^{i-1}}{1+q+\cdot+q^{i-1}}\leq
q^{\ell}\quad\text{($\ell$  integer $\geq 0$)}.
$$
This, combined with~\eqref{eq:distcr} and the relation
$\binom{n}{i-1}\leq \binom{n}{i}\leq \frac{n^i}{i!}$ which is valid
for $i\leq k\leq n/(2\,\log n)$ if $n$ is enough large, implies that
we have, for $j=1,\ldots,k-1$,
\begin{align*}
\left|B_{n,k,j}(q)\right| &\leq \sum_{i=0}^{k-j}
\frac{|1-q|^{i}}{[k-j-i]_{q}!}
 q^{\binom{k-j-i+1}{2}}\biggl( \binom{n}{i} q^{j}+\binom{n}{i-1}\biggr)\nonumber\\
&\leq  q^{k-j}(1+q^j)\sum_{i=0}^{k-j} \left(  \frac{n|1-q|}{q}
\right)^{i} \frac{1}{i!}\leq
q^{k-j}(1+q^j)\exp\left(\frac{n|1-q|}{q} \right).
\end{align*}
Equation~\eqref{eq:lemma-MomGF-1} is an immediate consequence of the
last inequality and the relation  ${q^{k-j}(1+q^j)\leq (1+2 q^k)}$
valid for $q>0$.

\textit{Step 2.} Let $u=\exp(t/\sigma_{n,k})$. We claim that, as $k\leq n/(2\,\log n)$ and $n\to\infty$,  we have
\begin{align}\label{eq:lemma-MomGF-2}
\sum_{j=1}^{k-1} \frac{[j]_u^n}{[j]_u!} u^{-kj}\leq k
\frac{[k-1]_u^n}{[k-1]_u!} u^{-k(k-1)}.
\end{align}

Set $r_{n,k,j}(u):=\frac{[j]_u^n}{[j]_u!} u^{-kj}$. Clearly, in
order to prove~\eqref{eq:lemma-MomGF-2}, it suffices to show that
$r_{n,k,j+1}(u)\geq r_{n,k,j}(u)$ for $j=1,\ldots,k-1$. We have, for
$1\leq j\leq k-1$,
\begin{align}\label{eq:lemma-MomGF-2b}
\frac{r_{n,k,j+1}(u)}{r_{n,k,j}(u)}=\frac{u^{-k}}{[j+1]_u}\left(1+\frac{u^j}{[j]_u}\right)^n
\geq \frac{u^{-k}}{[k]_u}\left(1+\frac{u^k}{[k]_u}\right)^n,
\end{align}
where the inequality follows from the relation  $q^{-j} [j]_q \leq
q^{-j-1} [j+1]_q$ ($q>0$). Using the asymptotic approximations
\eqref{eq:asympt u-int} and $
 u^k=1+tk/\sigma_{n,k}(1+o(1))$ (by definition, $u^k=e^{tk/\sigma_{n,k}}$),
it is easy to check that we have
\begin{align}
\frac{u^{-k}}{[k]_u}\left(1+\frac{u^k}{[k]_u}\right)^n &=\frac{1}{k}
\exp\left(\frac{n}{k}(1+o(1))\right)\to\infty .
\end{align}
This, combined with \eqref{eq:lemma-MomGF-2b}, implies that
$r_{n,k,j+1}(u)\geq r_{n,k,j}(u)$ for $1\leq j\leq k$ if $n$ is
enough large.

\textit{Conclusion.}
By~\eqref{eq:distcr}, $B_{n,k,k}(q)=q^{-k^2+k} [k]_q^n/[k]_q!$ for
all $n\geq k \geq 1$. Combining \eqref{eq:lemma-MomGF-1bis}
and~\eqref{eq:lemma-MomGF-2}, we arrive at
\begin{align}\label{eq:reste1}
\left|\frac{T_{n,k}(u)}{B_{n,k,k}(u)}-1\right| \leq k\left(1+2 \,u^k
\right) \exp\left(\frac{n|1-u|}{u} \right)[k]_u
\left(\frac{[k-1]_u}{[k]_u}\right)^n.
\end{align}
Let $R_{n,k}$ be the right-hand member of~\eqref{eq:reste1}.
Using~\eqref{eq:asympt u-int}, one can check that
\begin{align*}
n|1-u|/u&= n\,|\exp(-t/\sigma_{n,k})-1|= n\,|t|/\sigma_{n,k}(1+o(1))=o(n/k),\\
 \left(\frac{[k-1]_u}{[k]_u}\right)^n &= \left(1-\frac{1}{k}\,(1+o(1))\right)^n=\exp\left(- \frac{n}{k}(1+o(1))\right),
\end{align*}
from which it is easy to deduce that
\begin{align*}
R_{n,k}&=3\,k^2\,\exp\left(- \frac{n}{k}(1+o(1))\right)=o(1).
\end{align*}
Combining the latter approximation with \eqref{eq:reste1}, we
immediately obtain
\begin{align*}
T_{n,k}(u)=B_{n,k,k}(u)
\left(1+o\left(1\right)\right)=u^{-k^2+k}\frac{[k]_u ^n}{[k]_u!}
\left(1+o\left(1\right)\right),
\end{align*}
which is obviously equivalent to Lemma~\ref{lem:asymptotic-MomGF-1}
since $u^k=e^{tk/\sigma_{n,k}}=1+o(1)$. \qed

%\newpage
%%%%%%%%%%%%%%%%%%%%%%%%%%%%%%%%%%%%%%%%%%%%%%%%%%%%%%%%%%%%%%%%%%%%%%%%%%%%%%%%%%%%%%%%%%%%%%%%%%%%%%%%%%%%%%%%%%%%%%%%%%%%%%%%%
%%%%%%%%%%%%%%%%%%%%%%%%%%%%%%%%%%%%%%%%%%%%%%%%%%%%%%%%%%%%%%%%%%%%%%%%%%%%%%%%%%%%%%%%%%%%%%%%%%%%%%%%%%%%%%%%%%%%%%%%%%%%%%%%%
%%%%%%%%%%
%%%%%%%%%%                             The variance of X_n and Y_n
%%%%%%%%%%
%%%%%%%%%%%%%%%%%%%%%%%%%%%%%%%%%%%%%%%%%%%%%%%%%%%%%%%%%%%%%%%%%%%%%%%%%%%%%%%%%%%%%%%%%%%%%%%%%%%%%%%%%%%%%%%%%%%%%%%%%%%%%%%%%
%%%%%%%%%%%%%%%%%%%%%%%%%%%%%%%%%%%%%%%%%%%%%%%%%%%%%%%%%%%%%%%%%%%%%%%%%%%%%%%%%%%%%%%%%%%%%%%%%%%%%%%%%%%%%%%%%%%%%%%%%%%%%%%%%
\section{The variance  of $X_n$: proof of Theorem~\ref{thm:MomentsDistcrol}}

 This section is dedicated to proving Theorem~\ref{thm:MomentsDistcrol}.
Let us first recall the closed form expression for $\E(X_{n})$
recently obtained by the author in~\cite{KaMean}
%\begin{thm}(Kasraoui~\cite{KaMean})\label{thm:exactmean-crolcroc1}
%The mean $\E(X_{n})$ satisfies
\begin{align}
\E(X_{n})&=-\frac{5}{4}\frac{B_{n+2}}{B_{n}}+\left(\frac{n}{2}+\frac{9}{4}\right)\frac{B_{n+1}}{B_{n}}
+\frac{n}{2}+\frac{1}{4}.\label{eq:mean-crol}
\end{align}
%\end{thm}
 The proof of~\eqref{eq:mean-crol} in~\cite{KaMean} mainly relies on combinatorial arguments.
Note that~\eqref{eq:mean-crol} can also be extracted from the
generating function expansion~\eqref{eq:gf_crol}. We can go even
further and derive from~\eqref{eq:gf_crol} (or more directly by
relying on Proposition~\ref{prop:2nd derivative}) the  expression
\begin{align}
\begin{split}\label{eq:moments2-crol}
 \E\left(X_{n}^2-X_{n}\right)&=\frac{25}{16}\frac{B_{n+4}}{B_{n}} - \left(\frac{5}{4} n
+ \frac{407}{72}\right) \frac{B_{n+3}}{B_{n}}
 +\left(\frac{1}{4} n^{2}+ \frac{13}{12} n+\frac{223}{48} \right) \frac{B_{n+2}}{B_{n}} \\
&\quad +\left(\frac{1}{2} n^{2}-\frac{73}{18} \right)
\frac{B_{n+1}}{B_{n}}
 +\left(\frac{1}{4} n^{2}-\frac{1}{3} n -\frac{59}{144} \right).
\end{split}
\end{align}
A proof is given at the end of this section.
Combining the above formula with \eqref{eq:mean-crol} quickly
gives a compact expression for $\Var(X_{n})$ that we don't explicitly state
here due to lack of space. To analyze asymptotically this formula,
we shall use the approximations
\begin{align}
 &\frac{B_{n+s+t}}{B_{n+s}} =\left(\frac{n}{\log n}\right)^{t}
 \left (1+t\frac{\log\log n}{\log n} (1+o(1)) \right),\label{eq:Asymptotic_QuotientBell1}\\
&\frac{B_{n+s+t}}{B_{n+s}} -\frac{B_{n+t}}{B_{n}}=
s\,t\frac{n^{t-1}}{\left(\log n\right)^{t}} \left(1+t\frac{\log\log
n}{\log n}(1+o(1))\right),\label{eq:Asymptotic_QuotientBell2}
\end{align}
which are valid as $n\to\infty$ for any (fixed) integers $s$ and
$t$. These approximations can be derived from earlier work of Salvy
and Schakell~\cite{Sal}. For a proof
of~\eqref{eq:Asymptotic_QuotientBell1}, we refer the reader
to~\cite[Equation~(3.19)]{KaMean} and its proof there. Proof details
of~\eqref{eq:Asymptotic_QuotientBell2} are given at the end of this
section.

 It is now a routine matter to prove Theorem~\ref{thm:MomentsDistcrol}.
Inserting expressions~\eqref{eq:moments2-crol}
and~\eqref{eq:mean-crol} in the relation
$\Var(X_n)=\E(X_n^{2}-X_n)+\E(X_n)-\E(X_n)^{2}$, then using
approximation~\eqref{eq:Asymptotic_QuotientBell1}, little
rearrangement, gives (details are left to the reader)
\begin{align}
\Var(X_n)&= \frac{n^{2}}{4} \frac{B_{n+1}}{B_{n}} \left
(\frac{B_{n+2}}{B_{n+1}} -\frac{B_{n+1}}{B_{n}} \right)
        -\frac{5}{4}n \frac{B_{n+2}}{B_{n}} \left( \frac{B_{n+3}}{B_{n+2}} - \frac{B_{n+1}}{B_{n}} \right)\nonumber\\
       &\quad+\frac{25}{16} \frac{B_{n+2}}{B_{n}} \left(\frac{B_{n+4}}{B_{n+2}} - \frac{B_{n+2}}{B_{n}} \right)
        + \frac{9}{4}n\left(\frac{B_{n+2}}{B_{n}} -\left(\frac{B_{n+1}}{B_{n}}\right)^{2}\right)\label{eq:VarX_n_proof}\\
       &\quad +\frac{28}{12}n \frac{B_{n+2}}{B_{n}} +O\left(  \left( \frac{n}{\log n}\right)^3   \right).\nonumber
\end{align}
This, combined
with~\eqref{eq:Asymptotic_QuotientBell1}--\eqref{eq:Asymptotic_QuotientBell2},
immediately leads to Theorem~\ref{thm:MomentsDistcrol}. Before
closing this section, we give some proof details
of~\eqref{eq:moments2-crol} and~\eqref{eq:Asymptotic_QuotientBell2}.

%%%%%%%%%%%%%%%%%%%%%%%%%%%%%%%%%%%%
% new subsection
%%%%%%%%%%%%%%%%%%%%%%%%%%%%%%%%%%%%
\paragraph{\textit{Proof of~\eqref{eq:moments2-crol}}}
By the law of total expectation, we have
\begin{align}\label{eq:ConditionalExpectation^2_crol}
\E\big(X_{n}(X_{n}-1)\big)
 =\frac{1}{B_n}\sum_{k=1}^n S_{n,k} \E\big(X_{n,k}(X_{n,k}-1)\big).
\end{align}
The following result will enable us to ``simplify'' sums of the form
$\sum_{k=1}^nP(k)\,S_{n,k}$ for any polynomial~$P$.

\begin{lem}\label{lem:moments-stirling}
 For all integers $n,r\geq 0$, set $B^{(r)}_n:=\sum_{k=1}^{n} k^r S_{n,k}$.
Then we have: $B^{(r)}_n=\sum_{i=0}^r a_i^{(r)}B_{n+i}$, where the
family $\left(a_i^{(r)}\right)_{0\leq i\leq r}$ is defined
recursively by
$a_0^{(0)}=1$ and $a_i^{(r+1)}=a_{i-1}^{(r)}-\sum_{\ell=i}^r{r \choose
\ell}a_{i}^{(\ell)}$, with, by convention, $a_{-1}^{(r)}=0$ for $r\geq0$.
\end{lem}
Clearly, Lemma~\ref{lem:moments-stirling} is equivalent to the
relation $B^{(r+1)}_n=B^{(r)}_{n+1}-\sum_{\ell=0}^r {r\choose
\ell}B^{(\ell)}_n$,
which can be derived as follows:
\begin{align*}
 B^{(r+1)}_n=\sum_{k=1}^{n} k^r\, kS_{n,k}
            =B^{(r)}_{n+1}-\sum_{k=1}^{n} (k+1)^rS_{n,k}
             =B^{(r)}_{n+1}-\sum_{\ell=0}^r \sum_{k=1}^{n}  {r\choose \ell} k^{\ell}
             S_{n,k},
\end{align*}
where the second equality results from the identity
$kS_{n,k}=S_{n+1,k}-S_{n,k-1}$. The first values of the
$B_n^{(r)}:=\sum_{k=1}^{n} k^r S_{n,k}$ read
\begin{align}
\begin{split}\label{eq:momentsStirlings-first values}
&B^{(0)}_n=B_n,\quad B^{(1)}_n=B_{n+1}-B_{n},\quad B^{(2)}_n=B_{n+2}-2B_{n+1}\\
&B^{(3)}_n=B_{n+3}-3B_{n+2}+B_n,\quad
B^{(4)}_n=B_{n+4}-4B_{n+3}+4B_{n+1}+B_n.
\end{split}
\end{align}
 If we insert the expression of $\E\big(X_{n,k}(X_{n,k}-1)\big)$ given
in~Proposition~\ref{prop:2nd derivative} (recall that
$\E\big(X_{n,k}(X_{n,k}-1)\big)=T_{n,k}''(1)/S_{n,k}$) in \eqref{eq:ConditionalExpectation^2_crol},
and then use~\eqref{eq:momentsStirlings-first values} to simplify the resulting sum,
we painlessly arrive at~\eqref{eq:moments2-crol}.
\qed

%%%%%%%%%%%%%%%%%%%%%%%%%%%%%%%%%%%%
% new subsection
%%%%%%%%%%%%%%%%%%%%%%%%%%%%%%%%%%%%
%\subsection{Asymptotic approximation of $\Var(X_{n})$: proof of Theorem~\ref{thm:MomentsDistcrol}}
\paragraph{\textit{Proof of~\eqref{eq:Asymptotic_QuotientBell2}}}
All asymptotic in what follows are meant for $n\to\infty$. Our
demonstration relies on~\eqref{eq:Asymptotic_QuotientBell1} and the
approximation found by Salvy and Schakell (see Section~3.3
in~\cite{Sal})
\begin{align}
 &\frac{B_{n+u+2}}{B_{n+u}}
-\left(\frac{B_{n+u+1}}{B_{n+u}}\right)^{2}= \frac{n}{\left(\log
n\right)^{2}}\left(1+2\frac{\log\log n}{\log n} (1+o(1))
\right),\label{eq:Asymptotic_Bell_lem2}
\end{align}
valid for any fixed integer $u$. Note that if we multiply both sides
of Equation~\ref{eq:Asymptotic_Bell_lem2} by $B_{n+u}/B_{n+u+1}$,
then use the specialization of~\eqref{eq:Asymptotic_QuotientBell1}
at $t=-1$ in the right-hand side of the resulting equation, we get
the approximation
\begin{align}\label{eq:Asymptotic_QuotientBell_shift3}
\frac{B_{n+u+2}}{B_{n+u+1}} -\frac{B_{n+u+1}}{B_{n+u}}=
\frac{1}{\log n}\left(1+\frac{\log\log n}{\log n} (1+o(1)) \right).
\end{align}

 Let $s$ and $t$ be two nonnegative integers. It is easily
checked that
\begin{align}\label{eq:asymptotic-telescoping}
\frac{B_{n+s+t}}{B_{n+s}} -\frac{B_{n+t}}{B_{n}}
&=\sum_{i=0}^{s-1}\sum_{\ell=0}^{t-1}
\frac{B_{n+i+t}}{B_{n+i+1}}\left(\frac{B_{n+i+\ell+2}}{B_{n+i+\ell+1}}
-\frac{B_{n+i+\ell+1}}{B_{n+i+\ell}}\right).
\end{align}
Combining~\eqref{eq:Asymptotic_QuotientBell1}
and~\eqref{eq:Asymptotic_QuotientBell_shift3} shows that each of the
$st$ summands in~\eqref{eq:asymptotic-telescoping} is asymptotically
equal to $\frac{n^{t-1}}{\left(\log n\right)^{t}}\left
(1+t\frac{\log\log n}{\log n} (1+o(1)) \right)$, whence~\eqref{eq:Asymptotic_QuotientBell2}.
 \qed

%\newpage
%%%%%%%%%%%%%%%%%%%%%%%%%%%%%%%%%%%%%%%%%%%%%%%%%%%%%%%%%%%%%%%%%%%%%%%%%%%%%%%%%%%%%%%%%%%%%%%%%%%%%%%%%%%%%%%%%%%%%%%%%%%%%%%%%
%%%%%%%%%%%%%%%%%%%%%%%%%%%%%%%%%%%%%%%%%%%%%%%%%%%%%%%%%%%%%%%%%%%%%%%%%%%%%%%%%%%%%%%%%%%%%%%%%%%%%%%%%%%%%%%%%%%%%%%%%%%%%%%%%
%%%%%%%%%%%
%%%%%%%%%%%                 Limit laws of  $\croc$
%%%%%%%%%%%
%%%%%%%%%%%%%%%%%%%%%%%%%%%%%%%%%%%%%%%%%%%%%%%%%%%%%%%%%%%%%%%%%%%%%%%%%%%%%%%%%%%%%%%%%%%%%%%%%%%%%%%%%%%%%%%%%%%%%%%%%%%%%%%%%
%%%%%%%%%%%%%%%%%%%%%%%%%%%%%%%%%%%%%%%%%%%%%%%%%%%%%%%%%%%%%%%%%%%%%%%%%%%%%%%%%%%%%%%%%%%%%%%%%%%%%%%%%%%%%%%%%%%%%%%%%%%%%%%%%
\section{Limiting  distribution of $Y_{n,k}$}

The purpose of this section is to demonstrate
Theorems~\ref{thm:LimitDistcroc-ktoinfty} and~\ref{thm:MomentsDistcroc}. As we mentioned in the
introduction, our proof mainly relies on results about the
distribution of $X_{n,k}$ expounded in the introduction  and the
following result which quantifies the combinatorial ``closeness'' of
the parameters~$\croc$ and~$\crol$.
 \begin{lem}\label{lem:croc-crol}
For any set partition $\pi\in\Pi_n^k$, we have
$$
\crol(\pi)\leq \croc(\pi)\leq \crol(\pi)+2k(k-1).
$$
\end{lem}

\pf Let $\pi=\{B_1,B_2,\ldots,B_k\}$ be a partition in $\Pi_n^k$.
Clearly, if two arcs in the linear representation of $\pi$ cross,
then the corresponding chords in the circular representation of
$\pi$ cross. This proves that ${\croc(\pi)\geq \crol(\pi)}$.

For $i=1,2,\ldots,k$, let~$e_i$ be the chord in the circular
representation  of $\pi$ that joins $\min(B_i)$ and $\max(B_i)$ and
denote by $c(e_i)$ the number of chords in the  circular
representation of $\pi$ which cross with $e_i$. Then, it is easy to
check (see e.g. Figure~\ref{fig:representations}) that we have
$$\croc(\pi)\leq \crol(\pi)+\sum_{i=1}^k c(e_i).$$ To conclude the
proof, it suffices to show that $c(e_i)\leq 2(k-1)$ for
$i=1,\ldots,k$. This is due to the fact that the chord $e_{i}$ can
cross with at most two chords coming from the block $B_j$ for any
integer $j\neq i$ (see e.g. Figure~\ref{fig:representations}).
 \qed

%%%%%%%%%%%%%%%%%%%%%%%%%%%%%%%%%%%%%%%%%%%%%%%%%%%%%%%%%%%%%%%%%%%%%%%%
% new subsection
%%%%%%%%%%%%%%%%%%%%%%%%%%%%%%%%%%%%
\subsection{Limiting  distribution of $Y_{n,k}$} In this section,
\emph{all asymptotic are meant  for ${k=o(\sqrt{n})}$ and
$n\to\infty$  unless otherwise stated}. Note that
Lemma~\ref{lem:croc-crol} asserts that
\begin{align}\label{eq:croc-crol}
0\leq Y_{n,k}-X_{n,k}\leq 2k(k-1), \qquad\qquad (n\geq k\geq 1),
\end{align}
whence $\Var(Y_{n,k}-X_{n,k})=O\left(k^4\right)$. This, combined
with the approximation~\eqref{eq:asymptotic Var(Xnk)} of
$\Var(X_{n,k})$ and the well-known Cauchy-Schwartz  inequality
involving the covariance of two random variables~$U$ and~$V$
\begin{align}\label{eq:cov}
 |\Cov(U,V)|\leq \sqrt{\Var(U)\Var(V)},
\end{align}
leads, after a routine computation, to
\begin{align}\label{eq:varXnk-varYnk}
 \Var(Y_{n,k})
&=\Var(X_{n,k})+\Var(Y_{n,k}-X_{n,k})+2\, \Cov(X_{n,k}\,,Y_{n,k}-X_{n,k})\nonumber\\
&=\Var(X_{n,k})+ O\left({k}^3\,\sqrt{n}\right) =\frac{{k}^2-1}{12}\,
n +O\left({k}^3\,\sqrt{n}\right),
\end{align}
as stated in the first part of
Theorem~\ref{thm:LimitDistcroc-ktoinfty}.
 We now turn our attention to the second part of Theorem~\ref{thm:LimitDistcroc-ktoinfty}.
First note~that
\begin{align}\label{eq:decompo_croc}
\frac{Y_{n,k}-\E(Y_{n,k})}{\sqrt{\Var(Y_{n,k})}}
&=B_{n,k}\,\frac{X_{n,k}-\E(X_{n,k})}{\sqrt{\Var(X_{n,k})}}
+C_{n,k},
\end{align}
where $B_{n,k}$ and $C_{n,k}$ are the random variables defined on
$\Pi_{n,k}$  by
\begin{align*}
B_{n,k}=\sqrt{\frac{\Var(X_{n,k})}{\Var(Y_{n,k})}}\quad\text{and}\quad
C_{n,k}=\frac{\big(Y_{n,k}-X_{n,k}\big)-\E(Y_{n,k}-X_{n,k})}{\sqrt{\Var(Y_{n,k})}}.
\end{align*}
Since $B_{n,k}\overset{p}{\longrightarrow} 1$ and
$C_{n,k}\overset{p}{\longrightarrow}0$ (by~\eqref{eq:varXnk-varYnk}
and~\eqref{eq:croc-crol}) and $X_{n,k}$ is asymptotically Gaussian
(by Theorem~\ref{thm:LimitDistcrol-ktoinfty}), the second part of
Theorem~\ref{thm:LimitDistcroc-ktoinfty} is an immediate consequence
of~\eqref{eq:decompo_croc} and the following basic result in
probability theory:\newline
\noindent 
\emph{If $(A_n)_{n\geq 1}$, $(B_n)_{n\geq 1}$ and
$(C_n)_{n\geq 1}$ are sequences of random variables such that
${A_n\overset{d}{\longrightarrow}A}$,
$B_n\overset{p}{\longrightarrow}b$ and
$B_n\overset{p}{\longrightarrow}c$, where $b$ and $c$ are constant,
then $A_n B_n+ C_n\overset{d}{\longrightarrow} bA+c$.}

%%%%%%%%%%%%%%%%%%%%%%%%%%%%%%%%%%%%%%%%%%%%%%%%%%%%%%%%%%%%%%%%%%%%%%%%
% new subsection
%%%%%%%%%%%%%%%%%%%%%%%%%%%%%%%%%%%%
\subsection{An upper bound for the variance of $Y_{n}$: proof of Theorem~\ref{thm:MomentsDistcroc}}
The same reasoning as in the proof of~\eqref{eq:varXnk-varYnk} shows
that
\begin{align}\label{eq:majoration-VarYn}
 \Var(Y_{n})&\leq \Var(X_{n})+\Var(Y_{n}-X_{n})+2\sqrt{ \Var( X_{n})\Var(Y_{n}-X_{n})}.
\end{align}
Combining the above inequality with
Theorem~\ref{thm:MomentsDistcrol}, we see that
Theorem~\ref{thm:MomentsDistcroc} is an immediate corollary of the
following property:
\begin{equation}
\Var(Y_{n}-X_{n})=O\left(\frac{n^{4}}{\left(\log
n\right)^{4}}\right)\quad\text{ as
$n\to\infty$.}\label{eq:varcroc_lem}
\end{equation}
To prove~\eqref{eq:varcroc_lem}, first observe that, by the law of
total expectation, we have
\begin{align}\label{eq:ConditionalExpectation^2_croc-crol}
\Var(Y_{n}-X_{n})&\leq  \E\big(\left(Y_n-X_n\right)^2\big)
=\frac{1}{B_n}\sum_{k=1}^n\,S_{n,k}\,\E\big((Y_{n,k}-X_{n,k})^2\big)\nonumber\\
&\leq \frac{1}{B_n}\sum_{k=1}^n S_{n,k}\,2 k
(k-1)\,\E\big(Y_{n,k}-X_{n,k}\big),
\end{align}
where the last inequality is a consequence of~\eqref{eq:croc-crol}.
Using expressions~\eqref{eq:meanblock-crol}
and~\eqref{eq:meanblock-croc} for~$\E(X_{n,k})$ and $\E(Y_{n,k})$,
after a routine computation, we get
\begin{align*}
\E(Y_{n,k}-X_{n,k})&=\frac{5}{4}k(k-1)-\frac{3}{2}(n+1-k)\frac{S_{n,k-1}}{S_{n,k}}
 -\frac{1}{2}n(4n-5k+1)\frac{S_{n-1,k-1}}{S_{n,k}}\\
&\quad-10{n\choose 2}\frac{S_{n-2,k-2}}{S_{n,k}}+{n\choose 4}\frac{S_{n-4,k-2}}{S_{n,k}}\\
&\leq \frac{5}{4}k(k-1) +\frac{5}{2}n
k\frac{S_{n-1,k-1}}{S_{n,k}}+{n\choose
4}\frac{S_{n-4,k-2}}{S_{n,k}}.
\end{align*}
Inserting the latter relation
in~\eqref{eq:ConditionalExpectation^2_croc-crol} gives, after some
manipulations,
\begin{align}\label{eq:majorationExpectation^2_croc-crol}
 \Var(Y_n-X_n)
&\leq \frac{1}{B_n} \bigg(\frac{5}{2}B^{(4)}_{n}+ 5 n
B^{(3)}_{n-1}+2{n\choose 4}B^{(2)}_{n-4}\bigg),
\end{align}
where, as in Lemma~\ref{lem:moments-stirling},
$B^{(r)}_{n}=\sum_{k=1}^r k^r S_{n,k}$. Combining the above
inequality with the relations in~\eqref{eq:momentsStirlings-first
values} and the approximation~\eqref{eq:Asymptotic_QuotientBell1},
we finally arrive at
\begin{align*}
 \Var(Y_n-X_n)=O\left(\frac{B_{n+4}}{B_n}\right)
              =O\left(\frac{n^{4}}{\left(\log n\right)^{4}}\right),
\end{align*}
as stated in~\eqref{eq:varcroc_lem}. This concludes the proof of
Theorem~\ref{thm:MomentsDistcroc}.

%\newpage
%%%%%%%%%%%%%%%%%%%%%%%%%%%%%%%%%%%%%%%%%%%%%%%%%%%%%%%%%%%%%%%%%%%%%%%%%%%%%%%%%%%%%%%%%%%%%%%%%%%%%%%%%%%%%%%%%%%%%%%%%%%%%%%%%
%%%%%%%%%%%%%%%%%%%%%%%%%%%%%%%%%%%%%%%%%%%%%%%%%%%%%%%%%%%%%%%%%%%%%%%%%%%%%%%%%%%%%%%%%%%%%%%%%%%%%%%%%%%%%%%%%%%%%%%%%%%%%%%%%
%%%%%%%%%%%%%%%%%%%%%%%%%%%    RANGES OF THE PARAMETER $\croc$    %%%%%%%%%%%%%%%%%%%%%%%%%%%%%%%%%%%%%%%%%%%%%%%%%%
%%%%%%%%%%%%%%%%%%%%%%%%%%%%%%%%%%%%%%%%%%%%%%%%%%%%%%%%%%%%%%%%%%%%%%%%%%%%%%%%%%%%%%%%%%%%%%%%%%%%%%%%%%%%%%%%%%%%%%%%%%%%%%%%%
%%%%%%%%%%%%%%%%%%%%%%%%%%%%%%%%%%%%%%%%%%%%%%%%%%%%%%%%%%%%%%%%%%%%%%%%%%%%%%%%%%%%%%%%%%%%%%%%%%%%%%%%%%%%%%%%%%%%%%%%%%%%%%%%%
\section{Maximum values of the parameters $\crol$ and $\croc$}

 This section contains the proof of Theorems~\ref{thm:max-cro-bl} and~\ref{thm:max-cro}. 
 It seems difficult (in general) to determine $M_{n,k}^{(\ell)}$ and  $M_{n,k}^{(c)}$ directly 
 from their combinatorial definition. The key idea is to convert our original problem
 (find global maxima of functions defined on set partitions)  to a maximization problem
 of functions defined on integer partitions (these partitions are often easier to handle than set partitions).

 Recall that a  partition of a positive integer $n$
 is a finite nonincreasing sequence of positive integers $\la
=(\la_1,\la_2,\ldots,\la_k)$ such that $\la_1+\la_2+\cdots+\la_k=n$.
The $\la_i$ are called the parts of the partition. We often write
$\la=(1^{m_1}2^{m_2}3^{m_3}\cdots)$ where exactly $m_i$ of the
$\la_j$ are equal to $i$. It is usual to associate with a set
partition $\pi=B_1/B_2/\ldots/B_k$ its block-size vector $\bl(\pi)$,
which is the integer partition whose parts are $|B_1|$, $|B_2|$,
\ldots, $|B_k|$. For instance, if $\pi=1\,7/2\,3\,8/4/5\,6$, we have
$\bl(\pi)=(3,2,2,1)=(1\,2^2\,3)$. In the sequel, we let $\P(n,k)$
denote the set of all (integer) partitions of $n$ into exactly $k$
parts.

 For an integer partition $\la$ of $n$, set
\begin{align}\label{eq:Def-Ml and Mc}
M^{(\ell)}(\la)= \displaystyle \max_{\pi\in
\Pi_{n}:\,\bl(\pi)=\la}\crol(\pi) \quad\text{and}\quad M^{(c)}(\la)=
\displaystyle \max_{\pi\in \Pi_{n}:\,\bl(\pi)=\la}\croc(\pi).
\end{align}
 In Section~\ref{sec:Def-Ml and Mc}, we prove the following result.
\begin{thm}\label{thm:def-M}
If $\la=(1^{m_1}2^{m_2}3^{m_3}\cdots r^{m_r})$, we have
\begin{align}
 M^{(\ell)}(\la)&=\sum_{s=2}^{r} (2s-3)  \binom{m_s}{2}+\sum_{1\leq s< t\leq r} 2(s-1)m_sm_{t},\label{eq:def-weightMl}\\
  M^{(c)}(\la)&=\binom{m_2}{2}+\sum_{s=3}^{r} 2s  \binom{m_s}{2}+ 2 m_2 \sum_{t=3}^{r}m_{t}+\sum_{3\leq s< t\leq r}
  2sm_sm_{t}.
  \label{eq:def-weightMc}
\end{align}
\end{thm}
For every integers $n\geq k\geq 1$ and $u\in\{\ell,c\}$, let $\M_{n,k}^{(u)}$ be the set of the maximas
of $M^{(u)}(\la)$, where~$\la$ runs
over all partitions of $n$ into $k$ blocks, i.e.
\begin{align}\label{eq:def-Maximal set}
\M_{n,k}^{(u)}&=\{\la\in\P(n,k):\,M^{(u)}(\la)=\max(M^{(u)}(\tau):\,\tau\in
\P(n,k)).
\end{align}
Note that, by \eqref{eq:Def-Ml and Mc} and \eqref{eq:def-Maximal
set}, we have
\begin{align}\label{eq:MvsMset}
M_{n,k}^{(u)}=M_{n,k}^{(u)}(\la)\quad\text{for $u\in\{c,\ell\}$ and
every $\la\in\M_{n,k}^{(u)}$}.
\end{align}
In the sequel, we let $\la_{n,k}^{*}$ denote the (unique) partition
$(\la_1,\la_2,\ldots,\la_k)$ of $n$ such that $\left\lfloor
\frac{n}{k}\right\rfloor\leq \la_i\leq \left\lceil
\frac{n}{k}\right\rceil$ for $i=1,\ldots,k$.  For instance, we have
$\la_{7,3}^{*}=(2^2 3^1)=(3,2,2)$. The following result is proved in
Section~\ref{sec:Max-Ml and Mc}.

\begin{thm}\label{thm:max-M}
Suppose $n\geq k\geq 1$. Then,
\begin{enumerate}
 \item $\M_{n,k}^{(\ell)}=\{\la_{n,k}^*\}$.
 \item \begin{enumerate}
       \item  $\M_{n,k}^{(c)}=\{\la_{n,k}^*\}$ if $n\geq 3k$,
       \item  $\M_{3k-j,k}^{(c)}=\{\left(1^{s}2^{j-2s}3^{k-j+s}\right):\;s=\lfloor\tfrac{j}{2}\rfloor\}$
        if $0\leq j<k$ or $k\leq j\leq 2k-6$,
       \item  $\M_{k+j,k}^{(c)}=\{(1^{k-j}2^{j}),(1^{k-j+2}2^{j-4}3^{2})\}$ if $4\leq j\leq 5$ and $k\geq j$,
       \item  $\M_{k+j,k}^{(c)}=\{(1^{k-j}2^{j})\}$ if $0\leq j\leq 3$ and $k\geq j$.
       \end{enumerate}
\end{enumerate}
\end{thm}
Note that Equation~\eqref{eq:MvsMset}, in conjunction with the two above theorems, easily
leads to Theorem~\ref{thm:max-cro-bl}. For instance, if $n=qk+r$ with
$q=\lfloor\frac{n}{k}\rfloor$, we have $\la_{n,k}^*=\left(q^{k-r}(q+1 )^r\right)$.
If $q\geq 2$ (i.e., $n\geq 2k$), by \eqref{eq:MvsMset} and Theorems~\ref{thm:def-M} and~\ref{thm:max-M}, 
this implies that
\begin{align*}
M_{n,k}^{(\ell)}=M_{n,k}^{(\ell)}(\la_{n,k}^*)=(2q-3)\binom{k-r}{2}+(2q-1)\binom{r}{2}+
2(q-1)r(k-r).
\end{align*}
Using the relation $\binom{k-r}{2}+\binom{r}{2}+r(k-r)=\binom{k}{2}$
and replacing $q$ by $(n-r)/k$ in the above equality, we arrive at
Theorem~\ref{thm:max-cro-bl}(1a). 
The proof of the other assertions in Theorem~\ref{thm:max-cro-bl} are so similar that we leave the details to the reader.

In Section~\ref{sec:Max-Mln and Mcn}, we  deduce Theorem~\ref{thm:max-cro} from Theorem~\ref{thm:max-cro-bl}.
We conclude this section with remarks on the maximas of the parameter $\crol$ and $\croc$ 
in Section~\ref{sec:Max-comments}.

%\newpage
%%%%%%%%%%%%%%%%%%%%%%%%%%%%%%%%%%%%%%%%%%%%%%%%%%%%%%%%%%%%%%%%%%%%%%%%%%%%%%%%%%%%%%%%%%%%%%%%%%%%%%%%%%%%%%%%%%%%%%%%%%
\subsection{Proof of Theorem~\ref{thm:def-M}}\label{sec:Def-Ml and Mc}
%%%%%%%%%%%%%%%%%%%%%%%%%%%%%%%%%%%%%%%%%%%%%%%%%%%%%%%%%%%%%%%%%%%%%%%%%%%%%%%%%%%%%%%%%%%%%%%%%%%%%%%%%%%%%%%%%%%%%%%%%%

Using only~\eqref{eq:Def-Ml and Mc} and the combinatorial definition
of~$\crol$ and $\croc$, it is easy to compute $M^{(\ell)}(\la)$ and
$M^{(c)}(\la)$ for integer partitions $\la$ into two parts.
\fig2part
A look at Fig.~\ref{fig:2partition} and a little moment's thought
(we refer the reader to Sections~5 and 6 in~\cite{KaMean} for more details)
will convince the reader that
 \begin{align}
 \begin{split}\label{eq:cro-2partition}
  M^{(\ell)}(a,1)&=M^{(c)}(a,1)=0\;\;\text{if $a\geq 1$}, \\
  M^{(\ell)}(a,2)&=M^{(c)}(a,2)=2 \;\;\text{if $a\geq 3$},\quad M^{(\ell)}(2,2)=M^{(c)}(2,2)=1,\\
  M^{(\ell)}(a,b)&= 2(b-1)-\chi(a=b)\;\;\text{and}\;\;M^{(c)}(a,b)=2 b\;\;\text{if $a\geq b\geq 3$}.
  \end{split}
\end{align}
To go further, we will rely on the obvious fact that $\crol$ and
$\croc$ are \emph{$Z$-parameters} (see~\cite{KaMean}), i.e., for any
set partition $\pi=B_1/B_2/\cdots/B_k$,  we have
 \begin{align}\label{eq:cro-Zproperty}
 \crol(\pi)=\sum\crol \big(\st(B_i/B_j)\big)
 \quad\text{and}\quad
 \croc(\pi)=\sum\croc \big(\st(B_i/B_j)\big),
\end{align}
where the summations are over all pairs $(i,j)$ with $1\leq i<j\leq
k$ and $\st$ is the standardization map defined as follows. Recall that, given a
subset $S\subseteq \mathbb{P}$ with cardinality $|S|=n$, the
standardization map $\st$ is the (unique) order-preserving bijection
$\st: S \to [n]$.  We let $\st$ act element-wise on objects built
using $S$ as label. For instance, the set partition
$\pi=2\,9/4\,10/5/7\,11/8$ of~$S=\{2,4,5,7,8,9,10,11\}$ is sent
after standardization to the set partition
$\st(\pi)=1\,6/2\,7/3/4\,8/5$.

As it is easily seen (we omit the details), Theorem~\ref{thm:def-M}
is immediate from~\eqref{eq:cro-2partition} and the following lemma.
\begin{lem}\label{lem:M-Zproperty}
For every integer partition $\la=(\la_1,\la_2,\ldots,\la_k)$, we
have
 \begin{align}\label{eq:M-Zproperty}
 M^{(\ell)}(\la)=\sum_{1\leq i<j\leq k} M^{(\ell)}(\la_i,\la_j)
 \quad\text{and}\quad
 M^{(c)}(\la)=\sum_{1\leq i<j\leq k} M^{(c)}(\la_i,\la_j).
\end{align}
\end{lem}
 The proof of above result relies on
\eqref{eq:cro-Zproperty} and properties of certain set partitions
that we describe below. Note that it is immediate
from~\eqref{eq:cro-Zproperty} that $ M^{(u)}(\la)\leq \sum
M^{(u)}(\la_i,\la_j)$ for $u\in\{c,\ell\}$.

\textit{The set partitions $\pi(\la)$.} Suppose $\la=(\la_1,\la_2,
\ldots, \la_k)\in \P(n,k)$ and consider the Ferrers diagram $F$ of
$\la$ (which is an array of square cells having left-justified rows
with row $i$ containing $\la_i$ cells). Then, we put the integers
$1,2,\ldots,n$ in increasing order in the cells of $F$ "from top to
bottom and left to right", i.e., starting with the leftmost column,
filling its cells with the integers $1,2,\ldots,k$ from top to
bottom, then filling the next column with the integers
$k+1,k+2,\ldots,k+a_2$ where $a_2$ is the number of cells in the
second column of $F$ and working up to the right. Let $\pi(\la)$ be
the set partition of $[n]$ the blocks of which consist of the
elements in the same row of the filling of $F$.
As an example, the Ferrers diagram of $(4,2,1)$ and its corresponding filling are
\begin{center}
{\setlength{\unitlength}{1mm}
\begin{picture}(20,15)(0,0)
%%%%%%%%%%%%%%%%%%%%%%%%%%%%%%%%%%%%%%%%DIAGRAMME
 \put(0,0){\line(1,0){5}} \put(0,5){\line(1,0){10}} \put(0,10){\line(1,0){20}}\put(0,15){\line(1,0){20}}
\put(0,15){\line(0,-1){15}}\put(5,15){\line(0,-1){15}}\put(10,15){\line(0,-1){10}}
\put(15,15){\line(0,-1){5}}\put(20,15){\line(0,-1){5}}
\end{picture}}
\hspace{1cm} {\setlength{\unitlength}{1mm}
\begin{picture}(20,15)(0,0)
%%%%%%%%%%%%%%%%%%%%%%%%%%%%%%%%%%%%%%%%DIAGRAMME
 \put(0,0){\line(1,0){5}} \put(0,5){\line(1,0){10}} \put(0,10){\line(1,0){20}}\put(0,15){\line(1,0){20}}
\put(0,15){\line(0,-1){15}}\put(5,15){\line(0,-1){15}}\put(10,15){\line(0,-1){10}}
\put(15,15){\line(0,-1){5}}\put(20,15){\line(0,-1){5}}
%%%%%%%%%%%%%%%%%%%%%%%%%%%%%%%%%%%%%%%%REMPLISSAGE
\put(0,10){\makebox(6,4)[c]{\small 1}}
\put(0,5){\makebox(6,4)[c]{\small 2}}
\put(0,0){\makebox(6,4)[c]{\small 3}}
\put(5,10){\makebox(6,4)[c]{\small 4}}
\put(5,5){\makebox(6,4)[c]{\small 5}}
\put(10,10){\makebox(6,4)[c]{\small 6}}
\put(15,10){\makebox(6,4)[c]{\small 7}}
\end{picture}}
\end{center}
whence $\pi(4,2,1)=1\,4\,6\,7\,/\,2\,5\,/3$.

Suppose $\pi(\la)=B_1/B_2/\cdots/B_k$. Clearly, by definition of
$\pi(\la)$, we have $|B_i|=\la_i$ for $i=1,\ldots,k,$ and
$\st(B_i/B_j)=\pi(|B_i|,|B_j|)=\pi(\la_i,\la_j)$ for $1\leq i<j\leq
k$. Moreover, by construction of $\pi(a,b)$, for all integers $a\geq
b\geq 1$, we have
\begin{align*}
\pi(b,b)&=\{\{1,3,\ldots,2b-1\}\,,\,\{2,4,\ldots,2b\}\},\\
\pi(a,b)&=\{\{1,3,\ldots,2b-1,2b+1,2b+2,\ldots,a+b\}\,,\,\{2,4,\ldots,2b\}\}\quad\text{if
$a>b$}.
\end{align*}
For instance, $\pi(2,2)=1\,3/2\,4$ and $\pi(4,2)=1\,3\,5\,6/2\,4$.
\figpiab
  Using the above expressions
for $\pi(a,b)$, Fig.~\ref{fig:pi(a,b)}
and~\eqref{eq:cro-2partition}, it is easy to compute
$\crol(\pi(a,b))$ and $\croc(\pi(a,b))$ and check that
$\crou(\pi(a,b))=M^{(u)}(a,b)$ for $u\in\{c,\ell\}$ and all integers
$a\geq b\geq 1$ (details are left to the reader). To resume, we have
seen that
 $\bl(\pi(\la))=\la$ and, $\crou \left(\st(B_i/B_j)\right)=M^{(u)}(\la_i,\la_j)$ for $u\in\{c,\ell\}$
and $1\leq i<j\leq k$. By~\eqref{eq:cro-Zproperty}, this implies
that $ M^{(u)}(\la)\geq \sum M^{(u)}(\la_i,\la_j)$ for
$u\in\{c,\ell\}$. This ends the proof of Lemma~\ref{lem:M-Zproperty},
and thus completes the proof of Theorem~\ref{thm:def-M}.

%\newpage
%%%%%%%%%%%%%%%%%%%%%%%%%%%%%%%%%%%%%%%%%%%%%%%%%%%%%%%%%%%%%%%%%%%%%%%%%%%%%%%%%%%%%%%%%%%%%%%%%%%%%%%%%%%%
\subsection{Proof of Theorem~\ref{thm:max-M}}\label{sec:Max-Ml and Mc}
%%%%%%%%%%%%%%%%%%%%%%%%%%%%%%%%%%%%%%%%%%%%%%%%%%%%%%%%%%%%%%%%%%%%%%%%%%%%%%%%%%%%%%%%%%%%%%%%%%%%%%%%%%%%
For simplicity, we introduce auxiliary functions $R$, $T^{(\ell)}$
and $T^{(c)}$ defined for a partition
$\la=(\la_1,\la_2,\ldots,\la_k)=(1^{m_1}2^{m_2}3^{m_3}\cdots
r^{m_r})$ by
\begin{align}\label{eq:def-weight-RT}
\begin{split}
 R(\la)&:=\sum_{s=1}^{r} 2s \binom{m_s}{2}+\sum_{1\leq s< t\leq r}
 2sm_sm_{t}=\sum_{i=1}^k 2(i-1)\la_i,\\
 T^{(\ell)}(\la)&:=\sum_{s=2}^{r} \binom{m_s}{2}+2\binom{\ell(\la)}{2},\\
 T^{(c)}(\la)&:=2(m_1+m_2)\ell(\la)-2\binom{m_1+m_2+1}{2}+\binom{m_2}{2},
 \end{split}
\end{align}
where $\ell(\la)$ is the number of parts of $\la$. So, by
Theorem~\ref{thm:def-M}, we have
\begin{align}\label{eq:M vs R}
 M^{(\ell)}(\la)&=R(\la)-T^{(\ell)}(\la)
\quad\text{and}\quad M^{(c)}(\la)=R(\la)-T^{(c)}(\la).
\end{align}
The following result is the key ingredient in the proof of
Theorem~\ref{thm:max-M}.
\begin{lem}\label{lem:max1}
Let $\la=(\la_1,\la_2,\ldots,\la_k)=(1^{m_1}2^{m_2}\cdots r^{m_r})$
be a partition with two nonconsecutive parts $u,v$ with $u<v$ (i.e.,
$v-u\geq 2$ and $m_u m_v>0$). Let $\I_{u,v}^{+,-}(\la)$ be the
partition obtained from~$\la$ by decreasing the rightmost part of
$\la$ equal to $v$ by 1 and increasing the leftmost part of $\la$
equal to $u$ by 1. Then, if we set
$\widetilde{\la}=\I_{u,v}^{+,-}(\la)$, we have
\begin{align}
R(\tilde{\la})-R(\la)&=2\left(1+\sum_{u<t<v}m_t\right),\label{eq:maxR}\\
\begin{split}
T^{(\ell)}(\tilde{\la})-T^{(\ell)}(\la)&=2-m_u+m_{u+1}+m_{v-1}-m_{v}\\
                                   &\qquad+\chi(u+1=v-1)+(m_u-1)\chi(u=1),
\end{split}\label{eq:maxTl}\\
T^{(c)}(\tilde{\la})-T^{(c)}(\la)&=\left\{
                                       \begin{array}{ll}
                                         2(\ell(\la)-m_1)-1,      & \hbox{\text{if $u=1$, $v=3$};} \\
                                         m_2,                         & \hbox{\text{if $u=1$, $v\geq 4$};} \\
                                         -2(\ell(\la)-m_1)+m_2+1, & \hbox{\text{if $u=2$};}\\
                                         0,                           & \hbox{\text{if $u\geq 3$},}
                                       \end{array}
                                     \right.\label{eq:maxTc}
\end{align}
\end{lem}
\begin{proof} Set $d=\sum_{i=v}^r m_i$ and $s=1+\sum_{i=u+1}^r m_i$. By
definition of $\tilde{\la}$, we have
\begin{itemize}
 \item[(a)] $\tilde{\la}=(\tilde\la_1,\ldots,\tilde\la_k)$ with $\tilde
\la_d=\la_d-1$, $\tilde \la_s=\la_s+1$ and $\tilde \la_i=\la_i$ for
$i\neq d,s$;
\item[(b)] $\tilde{\la}=(1^{\tilde{m_1}}2^{\tilde{m_2}}\cdots
r^{\tilde{m_r}})$ with $\tilde{m_j}=m_j-1$ for $j\in\{u,v\}$,
$\tilde{m_{j}}=m_{j}+1+\chi(v=u+2)$ for $j\in\{u+1,v-1\}$  and
$\tilde{m_{j}}=m_{j}$ for $j\notin\{u,u+1,v-1,v\}$.
\end{itemize}
 Using expression
(a) for $\tilde{\la}$ and \eqref{eq:def-weight-RT}, we arrive at
$R(\tilde{\la})-R(\la)=\sum_{i=1}^{k} 2(i-1)(\tilde \la_i
-\la_i)=2(s-d)$. This proves the first assertion.  Similarly, the
other assertions can be obtained by using expression (b) for
$\tilde{\la}$ and~\eqref{eq:def-weight-RT}. We omit the details.
\end{proof}

 A very useful  consequence of Lemma~\ref{lem:max1} is the following result.
\begin{prop}\label{prop:notmaxima}
Suppose we are given integers $n,k\geq 1$ and let $\la$ be a
partition in $\P(n,k)$ which contains two nonconsecutive parts $u,v$
with $u<v$. Then, (1) $\la$ is not a maxima of $M^{(\ell)}$ on
$\P(n,k)$; (2) If in addition $(u,v)\neq (1,3)$, then $\la$ is not a
maxima of $M^{(c)}$ on $\P(n,k)$.
\end{prop}

\begin{proof} Set $\tilde{\la}=\I_{\,u,v}^{\,+,-}(\la)$, with
$\I_{\,u,v}^{\,+,-}(\la)$ defined as in Lemma~\ref{lem:max1}.
Observe that $\tilde{\la}\in\P(n,k)$. By~\eqref{eq:maxR}, we have
\begin{align}\label{eq:notmax1}
R(\tilde{\la})-R(\la)\geq 2+m_{u+1}+m_{v-1}.
\end{align}
This inequality, combined with~\eqref{eq:M vs R}
and~\eqref{eq:maxTl}, implies that
\begin{align*}
M^{(\ell)}(\tilde{\la})-M^{(\ell)}(\la) &\geq
m_u+m_{v}-\chi(v=u+2)-(m_u-1)\chi(u=1)\geq 1
\end{align*}
since $m_u$ and $m_v$ are positive integers. This proves the first
assertion.

Suppose now $(u,v)\neq(1,3)$ (and thus $v\geq 4$). Then, it is
immediate from~\eqref{eq:maxTc} that
$T^{(c)}(\tilde{\la})-T^{(c)}(\la)\leq m_2\chi(u=1)$. If we combined
this inequality with~\eqref{eq:notmax1} and~\eqref{eq:M vs R}, we
 arrive at
\begin{align*}
M^{(c)}(\tilde{\la})-M^{(c)}(\la) &\geq
2+m_{u+1}+m_{v-1}-\chi(v=u+2)-m_2\chi(u=1)\geq 2.
\end{align*}
This proves the second assertion. \end{proof}

 We now have enough tools to
verify Theorem~\ref{thm:max-M}.

\paragraph{\it{Proof of Theorem~\ref{thm:max-M}}}
It is easily checked that any partition $\la$ in $\P(n,k)$ which is
distinct from $\la_{n,k}^*$  contains  two nonconsecutive parts
$u,v$ such that  (i) $u<\left\lfloor\frac{n}{k}\right\rfloor<v$ or
 (ii) $u=\left\lfloor\frac{n}{k}\right\rfloor$ and
$v>\left\lceil\frac{n}{k}\right\rceil$. By
Proposition~\ref{prop:notmaxima}, this leads to assertions (1) and
(2a).

Let $\P(n,k;\leq 3)$ denote the set of all partitions in $\P(n,k)$
which have no part greater than $3$. If $n\leq 3k$, any partition in
$\P(n,k)\setminus\P(n,k;\leq 3)$ contains two parts $u,v$ with
$u<3<v$.  By Proposition~\ref{prop:notmaxima}(2), this implies 
that $\Mcnk\subseteq \P(n,k;\leq 3)$. 
Moreover, by~\eqref{eq:def-weightMc}, we have
$M^{(c)}(1^{a}2^{b}3^{d})=\binom{b}{2}+6\binom{d}{2}+2bd$, from
which it is easily seen that $
M^{(c)}(1^{a}2^{b+2}3^{d})<M^{(c)}(1^{a+1}2^{b}3^{d+1})$ and
$M^{(c)}(1^{a}2^{b+6})<M^{(c)}(1^{a+3}2^{b}3^{3})$ for all $a,b\geq
0$ and $d\geq 1$. Altogether, this shows that
\begin{align}\label{eq:M1}
\text{$k\geq \tfrac{n}{3}$ and $\la\in\Mcnk$}\Rightarrow
\text{$\la=\left(1^{a}2^{b}3^{c}\right)$ with $b\leq 1$ or ($c=0$
and $b\leq 5$)}.
\end{align}
On the other hand, (solving the system $\{a+b+c=k,a+2b+3c=3k-j\}$)
we see that, for $0\leq j\leq 2k$,
\begin{align}\label{eq:M2}
\P(3k-j,k;\leq3)&=\{\left(1^{s}2^{j-2s}3^{k-j+s}\right)\,:\;\max(0,j-k)\leq
s \leq\tfrac{j}{2}\}.
\end{align}
Combining~\eqref{eq:M1} with~\eqref{eq:M2} gives assertion~(2b).
This also shows that, if $j\in\{4,5\}$, $\M_{k+j,k}^{(c)}\subseteq
\{\left(1^{k-j}2^{j}\right),\left(1^{k-j+2}2^{j-4}3^{2}\right)\}$,
while, by~\eqref{eq:def-weightMc}, we have (if $j\in\{4,5\}$)
$M^{(c)}(1^{k-j}2^{j})= M^{(c)}(1^{k-j+2}2^{j-4}3^{2})$ . This gives assertion (2c). Assertion (2d) can be
obtained in a similar way. \qed

%\newpage
%%%%%%%%%%%%%%%%%%%%%%%%%%%%%%%%%%%%%%%%%%%%%%%%%%%%%%%%%%%%%%%%%%%%%%%%%%%%%%%%%%%%%%%%%%%%%%%%%%%%%%%%%%%%%%%%
\subsection{Proof of  Theorem~\ref{thm:max-cro}}\label{sec:Max-Mln and Mcn}
%%%%%%%%%%%%%%%%%%%%%%%%%%%%%%%%%%%%%%%%%%%%%%%%%%%%%%%%%%%%%%%%%%%%%%%%%%%%%%%%%%%%%%%%%%%%%%%%%%%%%%%%%%%%
It is a simple matter to derive Theorem~\ref{thm:max-cro} from
Theorem~\ref{thm:max-cro-bl}. By definition, we have
$M^{(u)}_{n}=\max\left(M^{(u)}_{n,k}:\,1\leq k\leq n\right)$ for $u\in\{\ell,c\}$.

We first show that $M^{(\ell)}_{n}=\left\lfloor \tfrac{1}{3}{n-1\choose 2}\right\rfloor$.
 For $1\leq n \leq 3$, $M^{(\ell)}_{n}=0$ and this assertion is
 true. 
 Suppose $n\geq 4$.
Then, use of the expressions given in Theorem~\ref{thm:max-cro-bl}
for $M^{(\ell)}_{n,k}$ and a basic study of the function $k\mapsto M^{(\ell)}_{n,k}$
(which is a piecewise function each part of which is a quadratic
function in $k$) the details of which are left to the reader shows
that
\begin{itemize}
 \item if $n\equiv 0 \pmod{3}$, the function $k\mapsto M^{(\ell)}_{n,k}$ has global maxima at exactly two points
 $k=\tfrac{n}{3}$ and $k=\tfrac{n}{3}+1$,
and the maximum is $M^{(\ell)}_{n,\tfrac{n}{3}}=M^{(\ell)}_{n,\tfrac{n}{3}+1}=\tfrac{n(n-3)}{6}=\left\lfloor
\tfrac{1}{3}{n-1\choose 2}\right\rfloor$;
\item if $n\equiv 1,2 \pmod{3}$, the function $k\mapsto M^{(\ell)}_{n,k}$ has a unique global maximum at $k=\left\lceil \tfrac{n}{3}\right\rceil$,
and the maximum is $M^{(\ell)}_{n,\left\lceil\tfrac{n}{3}\right\rceil}=\tfrac{1}{3}{n-1\choose 2}$.
\end{itemize}
 This concludes the proof of Theorem~\ref{thm:max-cro}(1). The
proof of the second part %\subsection{Computation of $M_{n}^{(c)}$}\label{sect:Mn-lin}
relies on the following result.
\begin{lem}\label{lem:croissance}
For $n\geq k\geq 1$, let $g_n(k)=(k-1)n-r_k(k-r_k)$, where $r_k$ is
the remainder of the division of $n$ by $k$. For $1\leq k\leq n-1$,
we have $g_{n}(k)<g_{n}(k+1)$.
\end{lem}
\begin{proof}  Let $q_{k+1}:=\lfloor\tfrac{n}{k+1}\rfloor$. Then,
noticing that $n=(k+1)q_{k+1}+r_{k+1}$ and using the definition
of~$g_n$, we obtain
\begin{align}\label{eq:croissance}
g_n(k+1)-g_n(k)=(k+1)q_{k+1}+r_{k}(k-r_{k})-r_{k+1}(k-r_{k+1}).
\end{align}
If $q_{k+1}\geq k$, this yields $g_{n}(k)<g_{n}(k+1)$ (since
$x(a-x)\leq a^2/4$ for $a,x\in\mathbb{R}$).

 Suppose $q_{k+1}< k$. Then, $q_{k+1}+r_{k+1}< 2k$, and since $n\equiv q_{k+1}+r_{k+1}\equiv r_k
\pmod{k}$, we have $r_{k+1}=ks+r_{k}-q_{k+1}$ with $s\in\{0,1\}$.
Inserting this identity in~\eqref{eq:croissance} gives, after a
routine computation,
$$
g_n(k+1)-g_n(k)=\left\{
  \begin{array}{ll}
    q_{k+1}(1+q_{k+1})+2r_{k}(k-q_{k+1}), & \hbox{if $s=0$;} \\
    q_{k+1}(1+q_{k+1})+2q_{k+1}(k-r_{k}), & \hbox{if $s=1$,}
  \end{array}
\right.
$$
from which it immediately results that
$g_{n}(k+1)-g_{n}(k)>q_{k+1}(1+q_{k+1})>0$.
\end{proof}

Theorem~\ref{thm:max-cro-bl}(2a), in conjunction with the above result,
implies that $M^{(c)}_{n,\lfloor\tfrac{n}{3}\rfloor}>M^{(c)}_{n,k}$
for $k< \lfloor\tfrac{n}{3}\rfloor$. On the other hand, using the
formulas (2b) and (2c)  in Theorem~\ref{thm:max-cro-bl}, a simple study of the sign
of $M^{(c)}_{n,k+1}-M^{(c)}_{n,k}$ for $k>\lceil\tfrac{n}{3}\rceil$,
shows that 
$M^{(c)}_{n,\lceil\tfrac{n}{3}\rceil}>M^{(c)}_{n,k}$ for
$k>\lceil\tfrac{n}{3}\rceil$. Finally, by comparing
$M^{(c)}_{n,\lfloor\tfrac{n}{3}\rfloor}$ and
$M^{(c)}_{n,\lceil\tfrac{n}{3}\rceil}$, we arrive at the following
result:
\begin{itemize}
 \item if $n\equiv 0 \pmod{3}$, the function $k\mapsto M^{(c)}_{n,k}$ has a unique global maximum at
 $k=n/3$ and the maximum is $M^{(c)}_{n,\tfrac{n}{3}}=6\binom{\tfrac{n}{3}}{2}=\left\lfloor
\tfrac{2}{3}{n-1\choose 2}\right\rfloor$;
 \item if $n\equiv 1\pmod{3}$, the function $k\mapsto M^{(c)}_{n,k}$ has global maxima at exactly two points
 $k=\left\lfloor \tfrac{n}{3}\right\rfloor$ and $k=\left\lceil \tfrac{n}{3}\right\rceil$,
and the maximum is $M^{(c)}_{n,\left\lfloor
\tfrac{n}{3}\right\rfloor}=M^{(c)}_{n,\left\lceil\tfrac{n}{3}\right\rceil}=6\binom{\lfloor\tfrac{n}{3}\rfloor}{2}=\left\lfloor
\tfrac{2}{3}{n-2\choose 2}\right\rfloor$;
\item if $n\equiv 2 \pmod{3}$, the function $k\mapsto M^{(c)}_{n,k}$ has a unique global maximum at
 $k=\left\lceil \tfrac{n}{3}\right\rceil$,
and the maximum of $f_n$ is $M^{(c)}_{n,\left\lceil\tfrac{n}{3}\right\rceil}
=6\binom{\lfloor\tfrac{n}{3}\rfloor}{2}+2\lfloor\tfrac{n}{3}\rfloor=\tfrac{2}{3}{n-2\choose 2}$.
\end{itemize}
 This completes the proof of Theorem~\ref{thm:max-cro}.

%%%%%%%%%%%%%%%%%%%%%%%%%%%%%%%%%%%%%%%%%%%%%%%%%%%%%%%%%%%%%%%%%%%%%%%%%%%%%%%%%%%%%%%%%%%%%%%%%%%%%%%%%%
%\newpage
%%%%%%%%%%%%%%%%%%%%%%%%%%%%%%%%%%%%%%%%%%%%%%%%%%%%%%%%%%%%%%%%%%%%%%%%%%%%%%%%%%%%%%%%%%%%%%%%%%%%%%%%%%
\subsection{Remarks on the maximas of the parameter $\crol$}\label{sec:Max-comments} 
This section presents results concerning the set partitions which maximize the parameter $\crol$.
We don't give proofs due to lack of space but the results announced can be painlessly 
deduced from Theorem~\ref{thm:max-M}
and properties of the set partitions $\pi(\la)$ described in Section~\ref{sec:Max-Ml and Mc}.

Let $a^{(\ell)}_{n,k}$ (resp., $a^{(\ell)}_{n}$) denote the number of partitions $\pi\in \Pi_n^k$ (resp., $\pi\in \Pi_n$)
satisfying $\crol(\pi)=M^{(\ell)}_{n,k}$  (resp.,
$\crol(\pi)=M^{(\ell)}_{n}$). Then, we have 
\begin{align*}
a^{(\ell)}_{n,k}=\left\{
          \begin{array}{ll}
             1, & \hbox{if $1\leq k\leq \left\lfloor \tfrac{n}{2}\right\rfloor $,} \\
             \binom{n}{2n-2k}, & \hbox{if $\left\lceil \tfrac{n}{2}\right\rceil \leq  k\leq n$;}
          \end{array}
         \right.
\;\textrm{and} \quad a^{(\ell)}_{n}=\left\{
          \begin{array}{ll}
             2, & \hbox{if $n\equiv 0 \pmod{3}$,} \\
             1, & \hbox{if $n\equiv 1,2 \pmod{3}$.}
          \end{array}
         \right.
\end{align*}

If $n\geq 4$ and $\left\lceil
\tfrac{n}{2}\right\rceil \leq  k\leq n$, the $\binom{n}{2n-2k}$
partitions $\pi\in\Pi_{n}^k$ satisfying
$\crol(\pi)=M^{(\ell)}_{n,k}$ are the
 partitions of $[n]$ which consist of exactly $n-k$ mutually disjoint and crossing arcs.
 If $n\geq 4$ and $1\leq k\leq \left\lfloor \tfrac{n}{2}\right\rfloor $, the (unique) partition~$\pi\in\Pi_{n}^k$ 
 satisfying $\crol(\pi)=M^{(\ell)}_{n,k}$ is the partition $\pi(\la_{n,k}^*)$ as defined in Section~\ref{sec:Max-Ml and Mc}.
 Furthermore,  
\begin{itemize}
\item if $n\equiv 0 \pmod{3}$, the two partitions $\pi\in\Pi_{n}$
satisfying $\crol(\pi)=M^{(\ell)}_{n}$ are~$\pi\big(\la_{n,\tfrac{n}{3}}^*\big)$
and~$\pi\big(\la_{n,\tfrac{n}{3}+1}^*\big)$;
\item  if $n\equiv 1,2 \pmod{3}$, the unique  partition $\pi\in\Pi_{n}$
satisfying $\crol(\pi)=M^{(\ell)}_{n}$ is~$\pi\big(\la_{n,\left\lceil
\tfrac{n}{3}\right\rceil}^*\big)$.
\end{itemize}

Finally, note that similar results for the  parameter $\croc$ should exist 
but seem much more complicated (except if $n\equiv 0 \pmod{3}$).

\medskip
{\bf Acknowledgements.} The idea for this article arose when the author was a post-doc at the Reykjak University in Winter 2010. 
   Furthermore, I am indebted to Christian Krattenthaler, who suggested me to use the generating function~\eqref{eq:gf_crol} to compute the variance
 of the parameter $\crol$, and to an anonymous referee for many suggestions which helped to improve the contents of this paper.
%%%%%%%%%%%%%

%%%%%%%%%%%%%%%%%%%%%%%%%%%%%%%%%%%%%%%%%%%%%%%%%%%%%%%%%%%%%%%%%%%%%%%%%%%%%%%%%%%%%%%%%%%%%%%%%%%%%%%%%%%%%%%%%%%%
%                                                   bibliography
%%%%%%%%%%%%%%%%%%%%%%%%%%%%%%%%%%%%%%%%%%%%%%%%%%%%%%%%%%%%%%%%%%%%%%%%%%%%%%%%%%%%%%%%%%%%%%%%%%%%%%%%%%%%%%%%%%%%

\end{document}